\documentclass[10pt,reqno]{amsart}
\usepackage{amssymb,amscd}   

\newtheorem{theorem}{Theorem}[section]
\newtheorem{proposition}[theorem]{Proposition} 
\newtheorem{corollary}[theorem]{Corollary}
\newtheorem{lemma}[theorem]{Lemma}
\newtheorem{remark}[theorem]{Remark}

\newtheorem{definition}[theorem]{Definition}

\numberwithin{equation}{section}

\begin{document}
\title[Non-tracial amalgamated free products]{Some analysis on amalgamated free products of \\ von Neumann algebras in non-tracial setting
}
\author[Y. Ueda]
{Yoshimichi UEDA}
\address{
Graduate School of Mathematics, 
Kyushu University, 
Fukuoka, 810-8560, Japan
}
\email{ueda@math.kyushu-u.ac.jp}

\maketitle

\begin{abstract} 
Several techniques together with some partial answers are given to the questions of factoriality, type classification and fullness for amalgamated free product von Neumann algebras. 
\end{abstract}  

\allowdisplaybreaks{

\section{Introduction} It was quite recent that the complete answers were given in \cite{Ueda:AdvMath11,Ueda:MRL} to the questions of factoriality, type classification, fullness and $\mathrm{Sd}$- and $\tau$-invariants for arbitrary free product von Neumann algebras. It is natural as a next project to consider the same questions for more general amalgamated free product von Neumann algebras. Such attempts were already made by us \cite{Ueda:PacificJMath99, Ueda:ASPM04, Ueda:TAMS03} almost 10 years ago for amalgamated free products over Cartan subalgebras. However the results there are far from satisfactory as compared to those on plain free product von Neumann algebras. The aim of this paper is to take a still very first step towards `satisfactory' answers to those questions for amalgamated free product von Neumann algebras. As simple consequences we will give partial answers at least when amalgamated free products are taken over type I von Neumann algebras, which are improvements of our previous works \cite{Ueda:PacificJMath99, Ueda:ASPM04, Ueda:TAMS03, Ueda:JFA05, Ueda:IllinoisJMath08}.  

The proofs in \cite{Ueda:AdvMath11,Ueda:MRL} are divided into analytical and combinatorial parts in essence. Combinatorial parts are completed by some `induction arguments', whose essential idea originates in several works due to Dykema, especially \cite{Dykema:DukeMathJ93}. On the other hand, analytical parts are devoted to proving several inequalities involving the Hilbert space norms arising from some states of particular form (instead of so-called free product states themselves), whose essential ideas apparently go back to the ICC argument for factoriality of group von Neumann algebras and the so-called $14\,\varepsilon$-argument both due to Murray and von Neumann. However our problems are of the nature of type III von Neumann algebras, and thus the lack of trace causes main difficulties. Hence the key is to overcome such difficulties. Here we will take up such analytical aspects in the general amalgamated free product setup, and indeed improve the analytical results in \cite{Ueda:AdvMath11,Ueda:MRL} with new techniques from the recent amazing development on type II$_1$ factors opened by several breakthroughs due to Popa. We hope that the technical facts provided in this paper are sufficient as analytical parts in future `best-possible' answers to the questions mentioned above at least in the case where amalgamated free products are taken over type I von Neumann subalgebras. 

The organization of this paper is as follows. Section 2 is preliminaries on amalgamated free product von Neumann algebras. In section 3 we provide a non-tracial version of one of the results in Ioana--Peterson--Popa's article \cite[Theorem 1.1]{IoanaPetersonPopa:Acta08}. In relation to it we provide a non-tracial adaptation of the so-called intertwining-by-bimidule criterion due to Popa, which may be of independent interest as future reference. In the same section we also generalize our previous results of controlling central sequences \cite[Proposition 3.5]{Ueda:AdvMath11},\cite[Proposition 3.1]{Ueda:MRL} to the amalgamated free product setting. In section 4, we give several partial answers to the questions mentioned above by utilizing technologies developed in \S3. Those include an answer to the factoriality and non-amenability questions of a given amalgamated free product $(M,E) = (M_1,E_1)\star_N (M_2,E_2)$ when $M_1$ is `diffuse relative to $N$', $M_2$ `non-trivial relative to $N$', and $N$ of type I.                    

Standard notation rule here follows our previous papers \cite{Ueda:AdvMath11, Ueda:MRL}; for example, the center, the unitary group and the set of projections of a given von Neumann algebra $M$ are denoted by $\mathcal{Z}(M)$, $M^u$ and $M^p$, respectively, and also the central support of $e \in M^p$ in $M$ by $c_e^M$. Notations and facts concerning amalgamated free products of von Neumann algebras will be summarized in next section 2.   

\section{Amalgamated Free Product von Neumann Algebras}

Let $M_1 \supseteq N \subseteq M_2$ be $\sigma$-finite von Neumann algebras, and faithful normal conditional expectations $E_1 : M_1 \rightarrow N$, $E_2 : M_2 \rightarrow N$ be given. Their amalgamated free product $(M,E) = (M_1,E_1)\star_N(M_2,E_2)$ is a pair of von Neumann algebra $M$ containing $M_1 \supseteq N \subseteq M_2$ and faithful normal conditional expectation $E : M \rightarrow N$ satisfying (i) $M = M_1\vee M_2$, (ii) $E\!\upharpoonright_{M_k} = E_k$ ($k=1,2$) and (iii) $E\!\upharpoonright_{\Lambda^\circ(M_1^\circ,M_2^\circ)} \equiv 0$, where $\Lambda^\circ(M_1^\circ,M_2^\circ)$ denotes the set of all alternating words in $M_1^\circ := \mathrm{Ker}(E_1)$ and $M_2^\circ := \mathrm{Ker}(E_2)$. The construction of such a pair is a bit complicated, but this simple formulation perfectly serves as a working definition. The construction was  introduced in the tracial setting in \cite{Popa:InventMath93} based on the $C^*$-algebraic one \cite{Voiculescu:LNM1132}. Its modular theoretical treatment was given in \cite{Ueda:PacificJMath99}, and will be reviewed below.  

Let $\chi$ be a faithful normal semifinite weight on $N$. Then the modular automorphism $\sigma_t^{\chi\circ E}$, $t \in \mathbb{R}$, is simply computed as 
\begin{equation}\label{Eq-2.1} \sigma_t^{\chi\circ E}\!\upharpoonright_{M_k} = \sigma_t^{\chi\circ E_k} \quad \quad  (k=1,2),\end{equation} 
see \cite[Theorem 2.6]{Ueda:PacificJMath99}. This formula together with famous Takesaki's criterion shows that for each $k=1,2$ there is a unique faithful normal conditional expectation $E_{M_k} : M \rightarrow M_k$ characterized by 
\begin{equation}\label{Eq-2.2}
E_{M_k}\!\upharpoonright_{\Lambda^\circ(M_1^\circ,M_2^\circ)\setminus M_k^\circ} \equiv 0. 
\end{equation} 
This fact is easily confirmed in the exactly same way as in \cite[Lemma 2.1]{Ueda:AdvMath11}. It is clear that $E\circ E_{M_k} = E$ holds. Consider the natural inclusion of the so-called continuous cores:   
\begin{equation}\label{Eq-2.3} 
\widetilde{M} := M\rtimes_{\sigma^{\chi\circ E}}\mathbb{R}\ \supseteq\ \widetilde{M}_k := M_k\rtimes_{\sigma^{\chi\circ E_k}}\mathbb{R}\ (k=1,2)\ \supseteq \widetilde{N} := N\rtimes_{\sigma^\chi}\mathbb{R},
\end{equation}     
which is independent of the choice of $\chi$ thanks to Connes's Radon-Nikodym cocycle theorem. The canonical liftings (still being faithful normal conditional expectations) $\widetilde{E} : \widetilde{M} \rightarrow \widetilde{N}$, $\widetilde{E}_k : \widetilde{M}_k \rightarrow \widetilde{N}$ ($k=1,2$) are constructed by 
\begin{equation}\label{Eq-2.4}
\widetilde{E} := E\bar{\otimes}\mathrm{Id}_{B(L^2(\mathbb{R}))}\!\upharpoonright_{M\rtimes_{\sigma^{\chi\circ E}}\mathbb{R}}, \quad\quad \widetilde{E}_k := E_k\bar{\otimes}\mathrm{Id}_{B(L^2(\mathbb{R}))}\!\upharpoonright_{M_k\rtimes_{\sigma^{\chi\circ E_k}}\mathbb{R}}.
\end{equation} 
Remark that the original $E$ and $E_k$ are recovered as the restrictions of $\widetilde{E}$ and $\widetilde{E}_k$ to $M$ and $M_k$ via the canonical embeddings $M \hookrightarrow \widetilde{M}$ and $M_k \hookrightarrow \widetilde{M}_k$, respectively.  Here is a simple but important fact \cite[Theorem 5.1]{Ueda:PacificJMath99} that $\widetilde{M}_1$ and $\widetilde{M}_2$ are freely independent with amalgamation over $\widetilde{N}$ with respect to $\widetilde{E}$, and moreover $\widetilde{M} = \widetilde{M}_1 \vee \widetilde{M}_2$. Consequently the following natural formula holds:
\begin{equation}\label{Eq-2.5}
(\widetilde{M},\widetilde{E}) = (\widetilde{M}_1,\widetilde{E}_1)\star_{\widetilde{N}}(\widetilde{M}_2,\widetilde{E}_2). 
\end{equation}
The canonical faithful normal semifinite traces $\mathrm{Tr}_{\widetilde{M}}$, $\mathrm{Tr}_{\widetilde{M}_k}$ ($k=1,2$) and $\mathrm{Tr}_{\widetilde{N}}$ on $\widetilde{M}$, $\widetilde{M}_k$ and $\widetilde{N}$, respectively, (see \cite[Theorem XII.1.1]{Takesaki:Book2}) must satisfy $\mathrm{Tr}_{\widetilde{M}}=\mathrm{Tr}_{\widetilde{N}}\circ\widetilde{E}$ and $\mathrm{Tr}_{\widetilde{M}_k}=\mathrm{Tr}_{\widetilde{N}}\circ\widetilde{E}_k$ (see e.g.~\cite[\S4]{Longo:CMP89}).  

Let $M^\omega \supseteq M_k^\omega\ (k=1,2)\ \supseteq N^\omega$ be the ultraproducts of $M \supseteq M_k\ (k=1,2)\ \supseteq N$. Here the inclusion relation is guaranteed by the existence of conditional expectations, and $E$ and $E_k$ ($k=1,2$) can be lifted up to $E^\omega : M^\omega \rightarrow N^\omega$ and $E_k^\omega : M_k^\omega \rightarrow N^\omega$, respectively. All the necessary facts on ultraproducts of von Neumann algebras are summarized in \cite[\S\S2.2]{Ueda:AdvMath11}. Remark that $M_1^\omega$ and $M_2^\omega$ are freely independent with amalgamation over $N^\omega$ with respect to $E^\omega$, see \cite[Proposition 4]{Ueda:TAMS03}. However it is hopeless due to \cite[Lemma 2.2]{Popa:IMRN07} that $M^\omega = M_1^\omega\vee M_2^\omega$ holds.       

\section{Technical Results} 

\subsection{A non-tracial adaptation of Popa's intertwining-by-bimodule criterion} Let $M$ be an arbitrary $\sigma$-finite (possibly type III) von Neumann algebra, and $A, B$ be its (possibly non-unital) von Neumann subalgebras with units $1_A, 1_B$, respectively. Suppose that $B$ is semifinite with a faithful normal semifinite trace $\mathrm{Tr}_B$ and furthermore that there is a faithful normal conditional expectation $E_B : 1_B M 1_B \rightarrow B$. 

\begin{proposition}\label{P-3.1} The following are equivalent{\rm:} 
\begin{itemize}
\item[(i)] There is no net $u_\lambda$ of unitaries in $A$ which satisfies $E_B(y^* u_\lambda x) \longrightarrow 0$ $\sigma$-strongly for any $x, y \in \bigcup\big\{ 1_A M p \,|\,p \in B^p; \mathrm{Tr}_N(p)< + \infty\}$. 
\item[(ii)] There are a normal {\rm(}possibly non-unital{\rm)} $*$-homomorphism $\rho : A \rightarrow M_n(\mathbb{C})\bar{\otimes}B$ with finite $n \in \mathbb{N}$ and a non-zero partial isometry $w \in M_n(\mathbb{C})\bar{\otimes}M$ such that 
\begin{itemize} 
\item $(\mathrm{Tr}_n\bar{\otimes}\mathrm{Tr}_B)(\rho(1_A)) < +\infty$, 
\item $ww^* \leq e_{11}\otimes 1_A$ and $w^* w \leq \rho(1_A)$, and 
\item $(e_{11}\otimes a)w = w\rho(a)$ for all $a \in A$. 
\end{itemize}  
\item[(iii)] There are non-zero projections $e \in A$, $f \in B$, a normal unital $*$-isomorphism $\theta : eAe \rightarrow fBf$ and a non-zero partial isometry $v \in M$ such that 
\begin{itemize} 
\item the central support $c_e^A$ is finite in $A$ and $\mathrm{Tr}_B(f) < +\infty$, 
\item $vv^* \leq e$ and $v^* v \leq f$, and 
\item $xv = v\theta(x)$ for all $x \in eAe$. 
\end{itemize} 
\end{itemize}
Suppose further that $M$ has an almost periodic weight $\psi$ such that both $A$ and $B$ sit inside the centralizer $M_\psi$, $\psi\!\upharpoonright_B$ is still semifinite, and the $E_B$ is the unique $\psi\!\upharpoonright_{1_B M 1_B}$-preserving one. Then the $w$ in {\rm(ii)} and the $v$ in {\rm(iii)} can be chosen in such a way that there is a common eigenvalue $\lambda$ of $\Delta_\psi$ so that $(\mathrm{id}_n\,\bar{\otimes}\,\sigma_t^\psi)(w) = \lambda^{it}w$ and $\sigma_t^\psi(v) = \lambda^{it}v$ for all $t \in \mathbb{R}$.
\end{proposition} 

As usual let us write $A \preceq_M B$ (with $E_B$ and $\mathrm{Tr}_B$) if the above equivalent conditions (i)--(iii) hold. Remark that no assumption on $A$ is necessary. The proof is of course modeled after Popa's original one for finite von Neumann algebras, but some cares are necessary. Indeed we observed this fact with $B$ finite several years ago, through our attempt to get better understanding of the fundamental articles \cite{Popa:AnnMath06, Popa:InventMath06} due to Popa. Houdayer and Vaes informed us that they have also observed it with $B$ finite independently (see \cite[Theorem 2.3]{HoudayerVaes:Preprint12}), and moreover Vaes corrected our misunderstanding on some argument in \cite[\S2]{ChifanHoudayer:DukeMathJ10}. The proof below is just a combination and/or a reformulation of several existing proofs of Popa's criterion \cite[Appendix]{Popa:AnnMath06},\cite[\S2]{Popa:InventMath06} (also see \cite[Appendix F]{BrownOzawa:Book},\cite[Appendix C]{Vaes:Asterisque07} for its exposition) and its variants \cite[\S3]{Asher:PAMS09},\cite[\S2]{ChifanHoudayer:DukeMathJ10},\cite[\S4]{Houdayer:Crelle09}, etc. The same idea as in e.g.~the proof of (1) $\Rightarrow$ (4) in \cite[Proposition C.1]{Vaes:Asterisque07} perfectly works for (ii) $\Rightarrow$ (i). (Note that the proof of (4) $\Rightarrow$ (1) in \cite[Theorem F.12]{BrownOzawa:Book} does not work at this point due to the lack of finite trace. Thus we could not prove (iii) $\Rightarrow$ (i) directly.) Hence the main parts below are (ii) $\Leftrightarrow$ (iii) and (i) $\Rightarrow$ (ii).

\medskip
{\it Proof of {\rm(ii)} $\Rightarrow$ {\rm(i):}} We may assume that $\rho(1_A)=\sum_{k=1}^n e_{kk}\otimes p_k$ with $p_k \in B^p$ thanks to \cite[Corollary 3.20]{Kadison:AmerJMath84}. Since $(\mathrm{Tr}_n\bar{\otimes}\mathrm{Tr}_B)(\rho(1_A)) < +\infty$, one has $w = \sum_{k=1}^n e_{1k}\otimes w_k$ with $w_k = w_k p_k \in \bigcup\big\{ 1_A M p \,|\,p \in B^p; \mathrm{Tr}_N(p)< + \infty\}$. On contrary, suppose that (i) is not true. One can find a net $u_\lambda$ in $A^u$ in such a way that $E_B(w_i^* u_\lambda w_j) \longrightarrow 0$ $\sigma$-strongly for all $i,j$, and hence $\rho(u_\lambda)(\mathrm{id}\bar{\otimes} E_B)(w^* w) = \sum_{i,j=1}^n e_{ij}\otimes E_B(w_i^* u_\lambda w_j) \longrightarrow 0$ $\sigma$-strongly. Therefore, $\Vert(\mathrm{id}_n\bar{\otimes} E_B)(w^* w)\Vert_{\mathrm{Tr}_n\bar{\otimes}\mathrm{Tr}_B} = \Vert\rho(u_\lambda)(\mathrm{id}_n\bar{\otimes} E_B)(w^* w)\rho(1_A)\Vert_{\mathrm{Tr}_n\bar{\otimes}\mathrm{Tr}_B} \longrightarrow 0$, a contradiction to $w\neq 0$. \qed  
  
\medskip
{\it Proof of {\rm(iii)} $\Rightarrow$ {\rm(ii):}} Since $v^* v \in \theta(eAe)'$, one can find a non-zero $z \in \mathcal{Z}(eAe)^p = (\mathcal{Z}(A)e)^p$ in such a way that the normal $*$-homomorphism $x \in zAz = (eAe)z \mapsto \theta(x)v^* v$ is injective. Since $c_e^A$ is finite in $A$, by \cite[Proposition 8.2.1]{KadisonRingrose:Book2} one can find non-zero, mutually orthogonal and equivalent (in $A$) $e_1,\dots,e_n \in A^p$ in such a way that $e_1 \leq z$ and $\sum_{k=1}^n e_k = c_{e_1}^A$. We have $e_1 vv^* \neq 0$, since $\theta(e_1)v^* v = v^* (e_1 vv^*)v$ by the choice of $z$ and $e_1 \leq z$. Then one gets partial isometries $v_1 := e_1, v_2,\dots,v_n \in A$ so that $v^*_k v_k = e_1$ and $v_k v_k^* = e_k$ ($k=2,\dots,n$). Since $e_1 Ae_1 \subseteq eAe$, we can construct a normal $*$-homomorphism $\rho : A \rightarrow M_n(\mathbb{C})\bar{\otimes}B$ by $\rho(a) := \sum_{i,j=1}^n e_{ij}\otimes\theta(v_i^* a v_j)$, $a \in A$. Set $w := \sum_{k=1}^n e_{1k}\otimes v_k v$ with $v$ in (iii), which defines a non-zero partial isometry, since $v^* v_i^* v_j v = \delta_{ij} v^* e_1 v$ and $v^* e_1 v = \theta(e_1)v^* v \neq 0$ as remarked before. Since $\sum_{i=1}^n v_i v_i^* = c_{e_1}^A = c_{e_k}^A$ for all $k=2,\dots,n$, we have $w\rho(a) = \sum_{i,j,k=1}^n e_{1i}e_{jk}\otimes v_i v\theta(v_j^* a v_k) = \sum_{i,k=1}^n e_{1i}e_{ik}\otimes v_i v_i^* a v_k v = \sum_{k=1}^n e_{1k}\otimes c_{e_k}^A a v_k v = (e_{11}\otimes a)w$ for all $a \in A$. Since $\rho(1_A) \leq 1_n\otimes f$, one has $(\mathrm{Tr}_n\bar{\otimes}\mathrm{Tr}_B)(\rho(1_A)) < +\infty$. \qed  

\medskip
{\it Proof of {\rm(ii)} $\Rightarrow$ {\rm(iii):}} As in (ii) $\Rightarrow$ (i) we may and do assume that $\rho(1_A) = \sum_{k=1}^n e_{kk}\otimes p_k$ with $\mathrm{Tr}_B$-finite $p_k \in B^p$. Note that any union of finite number of $\mathrm{Tr}_B$-finite projections is again $\mathrm{Tr}_B$-finite thanks to the Kaplansky formula \cite[Theorem 6.1.7]{KadisonRingrose:Book2}. Thus $p=\bigvee_{k=1}^n p_k$ is $\mathrm{Tr}_B$-finite, and replacing $B$ by $pBp$ (if necessary) we may and do assume that $\mathrm{Tr}_B(1_B) < +\infty$. Notice that $A$ must be of the form $A = A_0\oplus\mathrm{Ker}(\rho(-)w^*w)$ with $A_0$ finite, since $\rho(A)$ is finite. Note here that $w^* w \in \rho(A)'$, and thus $\rho(-)w^* w$ is a normal $*$-homomorphism.   

Let us first assume that $A_0$ has a type II$_1$ direct summand. By \cite[Lemma 6.5.6]{KadisonRingrose:Book2} one can find nonzero, mutually orthogonal and equivalent (in $A_0$) $e_1,\dots,e_n \in A_0^p$ whose sum is the unit of the type II$_1$ direct summand. With the center-valued trace $\tau : M_n(\mathbb{C})\bar{\otimes}B \rightarrow \mathbb{C}1\bar{\otimes}\mathcal{Z}(B)$ we have $n\tau(\rho(e_1)) \leq \tau(1\otimes1_B) = n\tau(e_{11}\otimes1_B)$, implying that there is a partial isometry $v_1 \in M_n(\mathbb{C})\bar{\otimes}B$ such that $v_1^* v_1 = \rho(e_1)$ and $v_1 v_1^* \leq e_{11}\otimes 1_B$. Since $v_1\rho(e_1)v_1^* = v_1 v_1^* \leq e_{11}\otimes1_B$, we can construct a normal unital $*$-isomorphism $\theta : eAe \rightarrow fBf$ with $e := e_1$, $f := \theta(e)$ in such a way that $e_{11}\otimes\theta(x) = v_1\rho(x)v_1^*$ for $x \in eAe$. Since $w^* w \in \rho(A)'\cap \rho(1_A)\big(M_n(\mathbb{C})\bar{\otimes}M\big)\rho(1_A)$ and $ww^* \in \big(\mathbb{C}e_{11}\bar{\otimes}A\big)'\cap (e_{11}\otimes1_A)\big(M_n(\mathbb{C})\bar{\otimes}M\big)(e_{11}\otimes1_A)$, it is easy to see that $wv_1^*$ is a non-zero partial isometry whose left and right support projections are less than $e_{11}\otimes e$ and $e_{11}\otimes f$, respectively, and hence $wv_1^* = e_{11}\otimes v$ for some non-zero partial isometry $v \in eMf$. Then one has $e_{11}\otimes xv = (e_{11}\otimes x)wv_1^* = w\rho(x)v_1^* = wv_1^* v_1\rho(x)v_1^* = e_{11}\otimes v\theta(x)$ for $x \in eAe$. 

We next consider the case that $A_0$ is of type I, that is, there is an abelian (in $A$) $e \in A_0^p$ with $c_e^A = 1_{A_0}$. With a MASA $\mathfrak{A}$ between $\rho(eAe)\oplus\mathbb{C}\rho(e)^\perp \subseteq M_n(\mathbb{C})\bar{\otimes}B$ one can choose, by \cite[Theorem 3.18]{Kadison:AmerJMath84}, mutually orthogonal and equivalent (in $M_n(\mathbb{C})\bar{\otimes}B$) projections $q_1,\dots,q_n$ from $\mathfrak{A}$ with $\sum_{k=1}^n q_k = 1_n\otimes1_B$. Then one immediately observes (by looking at their center-valued traces) that every $q_k$ is equivalent to $e_{11}\otimes1_B$ in $M_n(\mathbb{C})\bar{\otimes}B$. Since $\rho(e)w^*w \neq 0$, some $q:=q_k$ must satisfy $q\rho(e)w^*w \neq 0$. In this way, we can choose a non-zero partial isometry $v_1 \in M_n(\mathbb{C})\bar{\otimes}B$ in such a way $v_1^* v_1 =q\rho(e) (\leq \rho(e))$, $v_1 v_1^* \leq e_{11}\otimes 1_B$, and thus $v_1^* v_1 \in \rho(eAe)'$ and $wv_1^* \neq 0$ (since $q\rho(e)w^* w \neq 0$). Then we can construct a unital normal $*$-homomorphism $\theta : eAe \rightarrow fBf$ with $f := \theta(e)$ by $e_{11}\otimes\theta(x) = v_1\rho(x)v_1^*$ for $x \in eAe$ and a non-zero $y \in eMf$ by $e_{11}\otimes y = wv_1^*$. Moreover we have $e_{11}\otimes xy = (e_{11}\otimes x)wv_1^* = w\rho(x)v_1^* = wv_1^* (v_1\rho(x)v_1^*) = e_{11}\otimes y\theta(x)$ for $x \in eAe$, since $v_1^* v_1 \in \rho(eAe)'$. Hence $xy = y\theta(x)$ for  $x \in eAe$. Replacing $e$ by suitable $z \in \mathcal{Z}(eAe)^p$ (if necessary) we can make $\theta$ injective with keeping both $\theta(e) = f$ and $y=eyf$. With the polar decomposition $y=v|y|$ we get $vv^* \leq e$, $v^* v \leq f$ and $xv = v\theta(x)$ for $x \in eAe$. \qed

\medskip
We have two ways for completing the final part of the proof of (i) $\Rightarrow$ (ii) below; one is the use of Haagerup's $L^p$-space technologies and the other that of standard forms due to Araki, Connes and Haagerup. Here we use the latter as easy way. In what follows $(M \curvearrowright \mathcal{H}, J_M, \mathfrak{P}_M^\natural)$ denotes a standard form of $M$, see \cite[Definition IX.1.13]{Takesaki:Book2}.  

\medskip
{\it Proof of {\rm(i)} $\Rightarrow$ {\rm(ii):}} Note that $E_B(y^* u_\lambda x) \longrightarrow 0$ $\sigma$-strongly if and only if $\Vert E_B(y^* u_\lambda x)\Vert_{\mathrm{Tr}_B} \longrightarrow 0$ for any $x, y \in \bigcup\{1_A Mp\,|\,p\in B^p; \mathrm{Tr}_B(p) < +\infty\}$. Thus there are $\varepsilon>0$ and $\mathcal{F} \Subset \bigcup\{1_A Mp\,|\,p\in B^p; \mathrm{Tr}_B(p) < +\infty\}$ so that 
\begin{equation}\label{Eq-3.1} 
\sideset{}{_{x, y \in \mathcal{F}}}\sum \Vert E_B(y^* u x)\Vert_{\mathrm{Tr}_B} \geq \varepsilon \quad \text{for all $u \in A^u$}.
\end{equation} 
Each $x \in \mathcal{F}$ has a $\mathrm{Tr}_B$-finite $p_x \in B^p$ with $x = xp_x$, and $p := \bigvee_{x \in \mathcal{F}} p_x$ must be $\mathrm{Tr}_B$-finite as remarked in (ii) $\Rightarrow$ (iii). Thus, replacing $B$ by $pBp$ (if necessary) we may and do assume that $\mathrm{Tr}_B$ is a finite trace, that is, $\mathrm{Tr}_B(1_B) < +\infty$.

Choose a faithful normal state $\varphi_0$ on $1_B^\perp M1_B^\perp$, and set $\hat{B} := B\oplus\mathbb{C}1_B^\perp$ and $E_{\hat{B}} : x \in M \mapsto E_B(1_B x 1_B) + \varphi_0(1_B^\perp x1_B^\perp)1_B^\perp$ giving a faithful normal conditional expectation from the whole $M$ onto $\hat{B}$. Clearly $\hat{B}$ is still finite (since we have assumed that $\mathrm{Tr}_B$ is a finite trace), and the mapping $b+\alpha1_B^\perp \in \hat{B} \mapsto  \mathrm{Tr}_B(b)+\alpha \in \mathbb{C}$ defines a faithful normal trace (not weight !) $\mathrm{Tr}_{\hat{B}}$ on $\hat{B}$. Set $\varphi := \mathrm{Tr}_{\hat{B}}\circ E_{\hat{B}}$, a faithful normal positive linear functional on $M$, and let $\xi_0 \in \mathfrak{P}_M^\natural$ be its unique representing vector. It is standard, by a usual exhaustion argument like e.g.~the proof of \cite[Theorem IV.5.5]{Takesaki:Book1}, to see that there is a family of vectors $\{\xi_i\}_{i\in I}$ in $\mathcal{H}$ so that $\xi_0$ is in the family (thus $0$ is regarded as a distinguished element in $I$) and moreover $\mathcal{H} = \sum_{i\in I}^\oplus [J_M\hat{B}J_M\xi_i]$. Therefore, one can construct an isometry $U : \mathcal{H} \rightarrow \ell^2(I)\otimes L^2(\hat{B})$ satisfying $U\xi_0 = \delta_0\otimes\Lambda_{\mathrm{Tr}_{\hat{B}}}(1)$ and $U(J_M x^* J_M) = (1\otimes J_{\hat{B}} x^* J_{\hat{B}})U$ for $x \in \hat{B}$, where $L^2(\hat{B})$ is the usual standard Hilbert space constructed out of $\mathrm{Tr}_{\hat{B}}$, $\Lambda_{\mathrm{Tr}_{\hat{B}}}$ the canonical embedding of $\hat{B}$ to $L^2(\hat{B})$ and $J_{\hat{B}}$ the canonical unitary conjugation on $L^2(\hat{B})$. By the construction we observe that $P:=UU^* \in B(\ell^2(I))\bar{\otimes}\hat{B}$ and moreover that the pair $P\big(B(\ell^2(I))\bar{\otimes}\hat{B}\big)P$ and $P_{\mathbb{C}\delta_0}\otimes1$ with the rank $1$ projection $P_{\mathbb{C}\delta_0}$ onto $\mathbb{C}\delta_0$ is nothing but a concrete realization, modulo the unitary equivalence by $U$, of the basic extension $\langle M, \hat{B}\rangle$ and the Jones projection $e_{\hat{B}}$ associated with $E_{\hat{B}}$. Then 
\begin{equation}\label{Eq-3.2} \mathrm{Tr}_{\langle M, \hat{B}\rangle}(-) := (\mathrm{Tr}_{B(\ell^2(I))}\bar{\otimes}\mathrm{Tr}_{\hat{B}})(U(-)U^*)
\end{equation}
with the usual trace $\mathrm{Tr}_{B(\ell^2(I))}$ on $B(\ell^2(I))$ gives a faithful normal semifinite trace on the basic extension $\langle M, \hat{B}\rangle$. For $x \in \hat{B}$ one has $Uxe_{\hat{B}}U^* = P_{\mathbb{C}\delta_0}\otimes x$ and hence $\mathrm{Tr}_{\langle M, \hat{B}\rangle}(x e_{\hat{B}}) = 
(\mathrm{Tr}_{B(\ell^2(I))}\bar{\otimes}\mathrm{Tr}_{\hat{B}})(Ux e_{\hat{B}}U^*) = 
(\mathrm{Tr}_{B(\ell^2(I))}\bar{\otimes}\mathrm{Tr}_{\hat{B}})(P_{\mathbb{C}\delta_0}\otimes x) = \mathrm{Tr}_{\hat{B}}(x)$. 
Therefore, we get  
\begin{equation}\label{Eq-3.3} 
\mathrm{Tr}_{\langle M, \hat{B}\rangle}(x e_{\hat{B}} y) = \mathrm{Tr}_{\langle M, \hat{B}\rangle}(e_{\hat{B}}yxe_{\hat{B}}) = \mathrm{Tr}_{\langle M, \hat{B}\rangle}(\hat{E}_B(yx)e_{\hat{B}}) = 
\varphi(yx), \quad x,y \in M.
\end{equation}

Let $d := \sum_{y \in \mathcal{F}}ye_{\hat{B}}y^* \in \langle M, \hat{B}\rangle^+$, and then $\mathrm{Tr}_{\langle M, \hat{B}\rangle}(d) = \sum_{y \in \mathcal{F}} \varphi(y^* y) < +\infty$ by \eqref{Eq-3.3}. In the exactly same way as in the proof of (1) $\Rightarrow$ (2) of \cite[Theorem F.12]{BrownOzawa:Book} we see, by using \eqref{Eq-3.1}, that the $\sigma$-weakly closed convex hull $\mathfrak{C}$ of $\{ u^* d u\,|\,u \in A^u\}$ does not contain $0$. Moreover, it is plain to see that $J_M 1_B J_M d = d$. Since $1_B \in Z(\hat{B})$ and hence $J_M 1_B J_M \in Z(\langle M, \hat{B} \rangle)$, we conclude that $\mathfrak{C}$ sits in $(1_A J_M 1_B J_M)\langle M, \hat{B} \rangle(1_A J_M 1_B J_M)$. Since $d \geq 0$ and $\mathrm{Tr}_{\langle M, \hat{B} \rangle}(d) < +\infty$, $\mathfrak{C}$ is embedded, as a closed convex set, into $L^2(\langle M, \hat{B} \rangle,\mathrm{Tr}_{\langle M, \hat{B} \rangle})$, the usual GNS Hilbert space associated with $\mathrm{Tr}_{\langle M, \hat{B}\rangle}$. Hence one can choose a unique minimal point $d_0 \in \mathfrak{C}$ with respect to the Hilbert space norm $\Vert-\Vert_{\mathrm{Tr}_{\langle M,\hat{B}\rangle}}$, which in turn falls in $(1_A J_M 1_B J_M)\langle M, \hat{B} \rangle(1_A J_M 1_B J_M)\cap A'$ and satisfies $\mathrm{Tr}_{\langle M,\hat{B}\rangle}(d_0) < +\infty$. Choosing a suitable spectral projection of $d_0$   
we get a nonzero projection $e \in \langle M, \hat{B} \rangle\cap A'$ such that $e \leq 1_A J_M 1_B J_M$ and $\mathrm{Tr}_{\langle M, \hat{B}\rangle}(e) < +\infty$. 

The projection $e$ apparently gives an $A$--$B$ bimodule $\mathcal{K} := e\mathcal{H}$ with left and right (unital) actions $a\cdot\xi\cdot b := a J_M b^* J_M\xi$ for $a \in A$, $\xi \in \mathcal{K}$, $b \in B$. The GNS representation of $B$ associated with $\mathrm{Tr}_B$ is simply given by the restriction $B \curvearrowright L^2(B) := 1_B L^2(\hat{B})$ with the canonical embedding $\Lambda_{\mathrm{Tr}_B} := \Lambda_{\mathrm{Tr}_{\hat{B}}}\!\upharpoonright_B$, and moreover the canonical unitary conjugation $J_B$ is also just the restriction of $J_{\hat{B}}$ to $L^2(B)$. Thus we get the right $B$-module embedding $U_0 := U\!\upharpoonright_{\mathcal{K}} : \mathcal{K} \hookrightarrow \ell^2(I)\bar{\otimes}L^2(B)_B$ ($\subseteq (1\otimes 1_B)(\ell^2(I)\bar{\otimes}L^2(\hat{B}))$, and $U_0 U_0^* \in B(\ell^2(I))\bar{\otimes} B$ satisfies $(\mathrm{Tr}_{B(\ell^2(I))}\bar{\otimes}\mathrm{Tr}_B)(U_0 U_0^*) = \mathrm{Tr}_{\langle M, \hat{B}\rangle}(e) < +\infty$ by \eqref{Eq-3.2}. By the same reason as in the beginning of the proof of \cite[Proposition F.10]{BrownOzawa:Book} or by \cite[Lemma A.1]{Vaes:Asterisque07} there are $n \in \mathbb{N}$ and a nonzero $z \in \mathcal{Z}(B)^p$ such that $\mathcal{K}_0 := J_M z J_M\mathcal{K}$ is still a non-trivial $A$--$B$ bimodule and $(U_0\!\upharpoonright_{J_M z J_M\mathcal{K}})(U_0\!\upharpoonright_{J_M z J_M\mathcal{K}})^* = (1\otimes J_B z J_B)U_0 U_0^* (1\otimes J_B zJ_B) = (1\otimes z)U_0 U_0^* \precsim P_n\otimes z$ in $B(\ell^2(I))\bar{\otimes}B = (\mathbb{C}1\bar{\otimes}J_B BJ_B)'$, where $P_n$ is a rank $n$ projection in $B(\ell^2(I))$. Choose a partial isometry $v \in (\mathbb{C}1\bar{\otimes}J_B BJ_B)'$ with $v^* v=(U_0\!\upharpoonright_{J_M z J_M\mathcal{K}})(U_0\!\upharpoonright_{J_M z J_M\mathcal{K}})^*$ and $vv^* \leq P_n\otimes z$, and then we can define a right $B$-module embedding $V : \mathcal{K}_0 \hookrightarrow \mathbb{C}^n\bar{\otimes}L^2(B)$ by $V:=v(U_0\!\upharpoonright_{J_M z J_M\mathcal{K}})$ with a fixed identification $P_n\ell^2(I) = \mathbb{C}^n$. The embedding $V$ gives the normal (possibly non-unital) $*$-homomorphism $\rho : a \in A \mapsto VaV^* \in M_n(\mathbb{C})\bar{\otimes}B$. 

Let $\delta_i$ ($1\leq i\leq n$) be a standard basis of $\mathbb{C}^n$, and set $\xi_i := V^*(\delta_i\otimes\Lambda_{\mathrm{Tr}_B}(1_B)) \in \mathcal{K}_0$ ($1\leq i \leq n$). 
For $a \in A$, write $\rho(a) = \sum_{i,j=1}^n e_{ij}\otimes\rho(a)_{ij}$ with the matrix units $e_{ij}$ associated with the $\delta_i$, and then 
\begin{equation}\label{Eq-3.4} 
a\xi_j = \sideset{}{_{i=1}^n}\sum J_M\rho(a)_{ij}^* J_M\xi_i, \quad 1\leq j\leq n.
\end{equation} 
Consider $\mathbb{M} := M_{n+1}(\mathbb{C})\bar{\otimes}M \curvearrowright L^2(\mathbb{M}) := M_{n+1}(\mathbb{C})\bar{\otimes}\mathcal{H}$ (by left matrix-multiplication) with the canonical unitary conjugation $J_{\mathbb{M}}$ defined by $J_{\mathbb{M}}(e_{ij}\otimes\xi) := e_{ji}\otimes(J_M\xi)$ for $e_{ij}\otimes\xi \in L^2(\mathbb{M})$. The natural cone determined by $(\mathbb{M} \curvearrowright L^2(\mathbb{M}), J_{\mathbb{M}})$ is denoted by $\mathfrak{P}^\natural_{\mathbb{M}}$. Set $\hat{\xi} := \sum_{k=1}^n e_{0k}\otimes \xi_k \in L^2(\mathbb{M})$, and define a normal (possibly non-unital) $*$-homomorphism $\hat{\rho} : A \hookrightarrow \mathbb{M}$ by $\hat{\rho}(a) := e_{00}\otimes a + \sum_{i,j=1}^n e_{ij}\otimes\rho(a)_{ij}$ for $a \in A$. Here a standard matrix unit system $e_{ij}$ in $M_{n+1}(\mathbb{C})$ is indexed by $0,1,\dots,n$. By \eqref{Eq-3.4} one has $\hat{\rho}(a)\hat{\xi} = J_{\mathbb{M}}\hat{\rho}(a)^* J_{\mathbb{M}}\hat{\xi}$ for $a \in A$. A standard fact on polar decomposition in standard forms (c.f.~\cite[Exercise IX.1.2]{Takesaki:Book2},\cite[Lemma 3.1]{Asher:PAMS09}) guarantees the existence of a vector $|\hat{\xi}| \in \mathfrak{P}^\natural_{\mathbb{M}}$ and a partial isometry $\hat{w} \in \mathbb{M}$ satisfying that $\hat{w}|\hat{\xi}| = \hat{\xi}$, $\hat{w}^*\hat{w} = [\mathbb{M}'|\hat{\xi}|]$, $\hat{w}\hat{w}^* = [\mathbb{M}'\hat{\xi}]$ and $\hat{\rho}(a)\hat{w} = \hat{w}\hat{\rho}(a)$ for $a \in A$. Since $(e_{00}\otimes1_A)\hat{\xi} = \hat{\xi}$, one has $(e_{00}\otimes1_A)[\mathbb{M}'\hat{\xi}] = [\mathbb{M}'\hat{\xi}]$, and thus $\hat{w}\hat{w}^* \leq e_{00}\otimes1_A$. Here ($\rho(A) \subseteq$) $M_n(\mathbb{C})\bar{\otimes}M$ is naturally regarded as a corner of $\mathbb{M}$ by the numbering of the matrix units $e_{ij}$'s. Then one has, by \eqref{Eq-3.4} again, $J_{\mathbb{M}}\rho(1_A)J_{\mathbb{M}}\hat{\xi} = \hat{\xi}$, and hence $J_{\mathbb{M}}\rho(1_A)J_{\mathbb{M}}|\hat{\xi}| = |\hat{\xi}|$. By $J_{\mathbb{M}}|\hat{\xi}| = |\hat{\xi}| \in \mathfrak{P}^\natural_{\mathbb{M}}$ we get $\rho(1_A)[\mathbb{M}'|\hat{\xi}|] = [\mathbb{M}'|\hat{\xi}|]$ so that $\hat{w}^*\hat{w} \leq \rho(1_A) \leq \sum_{k=1}^n e_{kk}\otimes1_B$. Therefore, $\hat{w} = \sum_{k=1}^n e_{0k}\otimes w_k$ with $w_k \in 1_A M 1_B$. Letting $w := (e_{10}\otimes 1_A)\hat{w} = \sum_{k=1}^n e_{1k}\otimes w_k \in M_n(\mathbb{C})\bar{\otimes}M$ we have $w^* w \leq \rho(1_A)$, $ww^* \leq e_{11}\otimes1_A$ and $(e_{11}\otimes a)w = (e_{10}\otimes1_A)\hat{\rho}(a)\hat{w}=(e_{10}\otimes1_A)\hat{w}\hat{\rho}(a) = w\rho(a)$ for $a \in A$. We have assumed (by cutting by a projection in $B$) that $\mathrm{Tr}_B(1_B) < +\infty$, and hence $(\mathrm{Tr}_n\bar{\otimes}\mathrm{Tr}_B)(\rho(1_A)) < +\infty$ is now trivial. Hence we are done. \qed

\medskip
{\it Proof of the second part of the assertion{\rm:}} Only the proof of (i) $\Rightarrow$ (ii) needs small modification to prove this. Let us explain this in what follows. The standard form $(M\curvearrowright\mathcal{H},J_M,\mathfrak{P}_M^\natural)$ is constructed from $\psi$ so that $J_M \Delta_\psi J_M = \Delta_\psi^{-1}$. The $\mathrm{Tr}_B$ is given by $\psi\!\upharpoonright_B$. We need an extra argument in relation to the $d_0 \in (1_A J_M 1_B J_M)\langle M, \hat{B} \rangle(1_A J_M 1_B J_M)\cap A'$. By the assumption here the modular operator $\Delta_\psi$ has a diagonalization $\Delta_\psi = \sum_{\lambda>0} \lambda\,e^\psi_\lambda$ and satisfies $\Delta_\psi^{it} \in \langle M, \hat{B} \rangle \cap A'$ for all $t \in \mathbb{R}$. Hence all the $e^\psi_\lambda$'s fall in $\langle M, \hat{B}\rangle \cap A'$. Thus $e^\psi_\lambda\,d_0^{1/2}$ with some $\lambda$ defines a non-zero element in $\langle M, \hat{B} \rangle \cap A'$. Since $e^\psi_\lambda$ commutes with $1_A J_M 1_B J_M$ and since $\mathrm{Tr}_{\langle M, \hat{B}\rangle}(e^\psi_\lambda\,d_0\,e^\psi_\lambda) = \mathrm{Tr}_{\langle M, \hat{B}\rangle}(d_0^{1/2} e^\psi_\lambda d_0^{1/2}) \leq \mathrm{Tr}_{\langle M, \hat{B}\rangle}(d_0) < +\infty$, we may and do assume $d_0 = e^\psi_\lambda\,d_0\,e^\psi_\lambda$. Hence the $A$--$B$ bimodule $\mathcal{K}_0$ can be chosen as a subspace of $e^\psi_\lambda\,\mathcal{H}$. Therefore, the $\hat{\xi} \in L^2(\mathbb{M}) = M_{n+1}(\mathbb{C})\bar{\otimes}\mathcal{H}$ satisfies that $(I_{M_{n+1}(\mathbb{C})}\bar{\otimes}\,\Delta_\psi^{it})\hat{\xi} = \lambda^{it}\hat{\xi}$ for all $t \in \mathbb{R}$. Since $I_{M_{n+1}(\mathbb{C})}\bar{\otimes}\,\Delta_\psi$ is the modular operator of $\mathrm{Tr}_{n+1}\bar{\otimes}\,\psi$ on $\mathbb{M}$, $(I_{M_{n+1}(\mathbb{C})}\bar{\otimes}\,\Delta_\psi^{it})|\hat{\xi}|$ still falls in $\mathfrak{P}_{\mathbb{M}}^\natural$, see \cite[Lemma IX.1.4]{Takesaki:Book2}. Remark here that $J_{\mathbb{M}}$ there is nothing but the one constructed from $\mathrm{Tr}_{n+1}\bar{\otimes}\,\psi$. Hence, by the uniqueness of polar decomposition $(\mathrm{id}\bar{\otimes}\sigma_t^\psi)(\hat{w}) = \lambda^{it}\hat{w}$ and $(I_{M_{n+1}(\mathbb{C})}\bar{\otimes}\,\Delta_\psi^{it})|\hat{\xi}| = |\hat{\xi}|$ hold for every $t \in \mathbb{R}$. These modifications are enough to complete the proof. \qed 

\begin{remark}\label{R-3.2} {\rm Let $E_{\hat{B}}$ and $\varphi = \mathrm{Tr}_{\hat{B}}\circ E_{\hat{B}}$ be as in the proof of (i) $\Rightarrow$ (ii) above. Let $\widehat{E_{\hat{B}}} : \langle M, \hat{B}\rangle \rightarrow M$ be the dual operator-valued weight associated with $E_{\hat{B}}$ in the sense of \cite[\S\S1.2]{Kosaki:JFA86}. It is known that the modular operator $\Delta_\varphi$ and Connes's spacial derivative $(d(\varphi\circ\widehat{E_{\hat{B}}}))/(d(\mathrm{Tr}_{\hat{B}}\circ\mathrm{Ad}J_M((-)^*)))$ must coincide, see e.g.~the proof of \cite[Proposition 2.2]{IzumiLongoPopa:JFA98}. Moreover $\Delta_\varphi$ is affiliated with $\langle M, \hat{B} \rangle$, since $\varphi = \mathrm{Tr}_{\hat{B}}\circ E_{\hat{B}}$.  With these two facts one can prove that {\it the modular operator $\Delta_\varphi$ is the Radon--Nikodym derivative of $\varphi\circ\widehat{E_{\hat{B}}}$, i.e., $\varphi\circ\widehat{E_{\hat{B}}} = \mathrm{Tr}_{\langle M, \hat{B}\rangle, \Delta_\varphi}$ in the sense of} \cite[Lemma VIII.2.8]{Takesaki:Book2}. This explains, in full generality, the relationship that was pointed out in \cite[Eq.(1.3.1)]{Popa:InventMath06} in the almost periodic case.}
\end{remark} 

\subsection{A non-tracial version of Ioana--Peterson--Popa's theorem} Let us investigate an amalgamated free product $(M,E) = (M_1,E_1)\star_N(M_2,E_2)$.  

\begin{proposition}\label{P-3.3} Let $A$ be a {\rm(}unital{\rm)} von Neumann subalgebra of the centralizer $(M_1)_\varphi$ of a certain faithful normal state $\varphi$, and\, $\mathfrak{M}_1$ be a {\rm(}possibly non-unital{\rm)} dense {\rm(}in any von Neumann algebra topology{\rm)} $*$-subalgebra of $M_1$ with $E_1(\mathfrak{M}_1) \subseteq \mathfrak{M}_1$. Suppose that there is a net $v_\lambda$ of unitaries in $A$ such that $E_1(y^* v_\lambda x) \longrightarrow 0$ $\sigma$-strongly for all $x,y \in \mathfrak{M}_1$. Then any unitary $u \in M$ with $uAu^* \subseteq M_1$ must fall in $M_1$. In particular, $\mathcal{N}_M(A) = \mathcal{N}_{M_1}(A)$ and $A'\cap M = A'\cap M_1$. Here $\mathcal{N}_P(Q)$ denotes the set of unitaries $u \in P$ with $uQu^* = Q$ for a given unital inclusion $P \supseteq Q$ of von Neumann algebras. 
\end{proposition} 

This is nothing but a non-tracial version of \cite[Theorem 1.1]{IoanaPetersonPopa:Acta08} due to Ioana, Peterson and Popa. Although the proof below is modeled after their proof, we need to overcome some difficulties due to the lack of trace by utilizing modular theoretic technologies. 

\begin{proof} Let $(M \curvearrowright \mathcal{H}, J_M, \mathfrak{P}^\natural_M)$ be a standard form of $M$, and $\xi_0 \in \mathfrak{P}^\natural_M$ be the unique representing vector of $\varphi\circ E_{M_1}$. Let $e_{M_1}$ be the so-called Jones projection associated with $E_{M_1}$, i.e., $e_{M_1}x\xi_0 = E_{M_1}(x)\xi_0$ for $x \in M$, and the basic extension $\langle M, M_1 \rangle$ is defined to be $M \vee \{e_{M_1}\}'' = J_M M_1' J_M \curvearrowright \mathcal{H}$. Consider the projection $p := [A J_M M_1 J_M u^*\xi_0] \in A' \cap(J_M M_1 J_M)' = A' \cap \langle M, M_1\rangle$. Notice that $a J_M x^* J_M u^*\xi_0 = J_M x^* J_M u^*(uau^*)\xi_0$ for $a \in A$ and $x \in M_1$, and moreover that $uau^* \in M_1$ can be approximated in any von Neumann algebra topology, by analytic elements, say $y_\lambda$, in $M_1$ with respect to the modular action $\sigma^\varphi$. Those altogether show that 
\begin{align*} 
a J_M x^* J_M u^*\xi_0 = \lim_\lambda J_M x^* J_M u^* y_\lambda\xi_0 = \lim_\lambda J_M x^* \sigma_{i/2}^\varphi(y_\lambda)^* J_M u^*\xi_0 \in [J_M M_1 J_M u^*\xi_0]
\end{align*}    
thanks to $\sigma_t^{\varphi\circ E_{M_1}}\!\upharpoonright_{M_1} = \sigma_t^\varphi$ ($t \in \mathbb{R}$) and \cite[Lemma VIII.3.18 (ii)]{Takesaki:Book2}. Consequently we get $p \leq [J_M M_1 J_M u^*\xi_0] = u^*e_{M_1}u$, which and $\widehat{E_{M_1}}(e_{M_1}) = 1$ imply $\Vert\widehat{E_{M_1}}(p)\Vert_\infty < +\infty$, where $\widehat{E_{M_1}} : \langle M, M_1 \rangle \rightarrow M$ denotes the dual operator-valued weight of $E_{M_1}$. See \cite[\S\S1.2, Lemma 3.1]{Kosaki:JFA86}. We will prove $(1-e_{M_1})p(1-e_{M_1}) = 0$. In fact, if this is the case, then $p \leq e_{M_1}$ so that $u^*\xi_0 = e_{M_1}u^*\xi_0 = E_{M_1}(u^*)\xi_0$, implying $u = E_{M_1}(u) \in M_1$ since $\xi_0$ is separating for $M \curvearrowright \mathcal{H}$. Since $\Vert \widehat{E_{M_1}}(p)\Vert_\infty < +\infty$ and $\widehat{E_{M_1}}(e_{M_1}) = 1$ as before, any spectral projection $f$ of $(1-e_{M_1})p(1-e_{M_1})$ corresponding to $[\delta,1]$ with arbitrary $\delta > 0$ still satisfies $\Vert\widehat{E_{M_1}}(f)\Vert_\infty < +\infty$. Therefore, it suffices to prove that any projection $f \in A'\cap\langle M, M_1 \rangle$ satisfying both $f \leq 1-e_{M_1}$ and $\Vert\widehat{E_{M_1}}(f)\Vert_\infty < +\infty$ must be $0$. 

In what follows we denote by $\mathcal{A}$ the $*$-subalgebra of $M$ consisting of all analytic elements with respect to $\sigma^{\varphi\circ E_{M_1}}$, which is well-known to be dense in any von Neumann algebra topology. Set $\psi := \varphi\circ E_{M_1}\circ \widehat{E_{M_1}}$, a faithful normal semifinite weight on $\langle M, M_1 \rangle$, and let $\langle M, M_1 \rangle \curvearrowright L^2(\langle M, M_1 \rangle, \psi)$ be the GNS representation with canonical embedding $\Lambda_{\psi} : \mathfrak{n}_\psi := \{x \in \langle M, M_1 \rangle\,|\,\psi(x^* x) < +\infty\} \rightarrow L^2(\langle M, M_1 \rangle, \psi)$ and norm $\Vert-\Vert_\psi$ associated with the weight $\psi$. Remark that $E_{M_1}(\mathcal{A}) \subseteq \mathcal{A}$ (thanks to $E_{M_1}\circ\sigma_t^{\varphi\circ E_{M_1}} = \sigma_t^\varphi\circ E_{M_1}$ for all $t \in \mathbb{R}$) and thus $\mathrm{span}(\mathcal{A}e_{M_1}\mathcal{A})$ becomes a dense (in any von Neumann algebra topology) $*$-subalgebra of $\mathfrak{n}_\psi^*\cap\mathfrak{n}_\psi$, and hence $\Lambda_\psi(\mathrm{span}(\mathcal{A}\,e_{M_1}\mathcal{A}))$ is dense in $L^2(\langle M, M_1\rangle,\psi)$ by \cite[Lemma 2.1]{IzumiLongoPopa:JFA98}. Thus one can choose a sequence $T_n \in \mathrm{span}(\mathcal{A}\,e_{M_1}\mathcal{A})$ in such a way that 
$\Vert \Lambda_\psi(T_n - f)\Vert_\psi \longrightarrow 0$ as $n \rightarrow \infty$, where note that $f$ clearly falls in $\mathfrak{n}_\psi$. Since $f \leq 1-e_{M_1}$ and $\sigma_t^\psi(e_{M_1}) = e_{M_1}$ ($t \in \mathbb{R}$) \cite[Lemma 5.1]{Kosaki:JFA86}, we also have $\Vert \Lambda_\psi((1-e_{M_1})T_n(1-e_{M_1})-f)\Vert_\psi \longrightarrow 0$ as $n \rightarrow \infty$ so that may and do assume that $T_n = (1-e_{M_1})T_n(1-e_{M_1})$ for all $n$. 

On contrary, suppose $f \neq 0$, that is, $\gamma := \Vert\Lambda_\psi(f)\Vert_\psi \gneqq 0$. Then one can choose $T := T_{n_0} \in \mathrm{span}(\mathcal{A}e_{M_1}\mathcal{A})$ with some $n_0$ in such a way that 
\begin{equation}\label{Eq-3.5} 
\Vert\Lambda_\psi(T)\Vert_\psi \leq 3\gamma/2, \quad 
\Vert\Lambda_\psi(T - f)\Vert_\psi \leq \gamma/5. 
\end{equation}   
For any $v \in A^u$ we compute 
\begin{align*} 
\gamma^2 - |\psi(T^* vTv^*)| 
&\leq 
|\psi(fvfv^*) - \psi(T^* vTv^*)| \\
&\leq 
|\psi((f-T)^* vfv^*)| + |\psi(T^* v(f-T)v^*)| \\
&\leq 
\Vert\Lambda_\psi(f-T)\Vert_\psi \Vert\Lambda_\psi(vfv^*)\Vert_\psi + \Vert\Lambda_\psi(T)\Vert_\psi \Vert\Lambda_\psi(v(f-T)v^*)\Vert_\psi \\
&\leq 
\Vert\Lambda_\psi(f-T)\Vert_\psi \Vert\Lambda_\psi(f)\Vert_\psi + \Vert\Lambda_\psi(T)\Vert_\psi \Vert\Lambda_\psi(f-T)\Vert_\psi \\
&\leq \gamma^2/2, 
\end{align*}  
where the first, the third, the fourth and the fifth inequalities follow from $f \in A'\cap\langle M,M_1\rangle$, the Cauchy--Schwarz inequality, $v \in (M_1)_\varphi \subset \langle M, M_1\rangle_\psi$, and \eqref{Eq-3.5}, respectively. Therefore, $\gamma^2 \leq 2|\psi(T^* v T v^*)|$ holds for all $v \in A^u$. Since $T = (1-e_{M_1})T(1-e_{M_1})$, we can write $T = \sum_{k=1}^m x_k e_{M_1} y_k$ with $x_k, y_k \in \mathcal{A}\cap\mathrm{Ker}(E_{M_1})$. Thus, for every $v \in A^u$ we have 
\begin{align*} 
\gamma^2 
&\leq 
2\sideset{}{_{k,l=1}^m}\sum|\psi(y_k^* e_{M_1} x_k^* v x_l e_{M_1} y_l v^*)| \\
&=
2\sideset{}{_{k,l=1}^m}\sum|\psi(y_k^* E_{M_1}(x_k^* v x_l)e_{M_1} y_l v^*)| \\ 
&= 
2\sideset{}{_{k,l=1}^m}\sum|\varphi\circ E_{M_1}(y_k^* E_{M_1}(x_k^* v x_l)y_l v^*)| \\
&= 
2\sideset{}{_{k,l=1}^m}\sum|\varphi\circ E_{M_1}(\sigma_i^{\varphi\circ E_{M_1}}(y_l)v^* y_k^* E_{M_1}(x_k^* v x_l))| \\
&\leq 
2 \max_{1\leq k \leq m}\Vert y_k\Vert_\infty \max_{1\leq l \leq m} \Vert\sigma_i^{\varphi\circ E_{M_1}}(y_l)\Vert_\infty \sideset{}{_{k,l=1}^m}\sum \Vert E_{M_1}(x_k^* v x_l)\Vert_\varphi. 
\end{align*}
Here the third equality is due to $\widehat{E_{M_1}}(e_{M_1}) = 1$, the fourth one follows from $v \in (M_1)_\varphi \subseteq M_{\varphi\circ E_{M_1}}$ and $y_l \in \mathcal{A}$ with the so-called modular condition, and finally the last inequality is due to the Cauchy--Schwarz inequality. Consequently we have chosen $x_1,\dots,x_m \in \mathcal{A}\cap\mathrm{Ker}(E_{M_1})$ and a universal constant $C > 0$ so that 
\begin{equation}\label{Eq-3.6} 
\gamma^2 \leq C \sideset{}{_{k,l=1}^m}\sum \Vert E_{M_1}(x_k^* v x_l)\Vert_\varphi \quad \text{for all $v \in A^u$}.  
\end{equation}     

Set $\mathfrak{M}_1^\circ := \mathfrak{M}_1 \cap M_1^\circ$. By the assumption on $\mathfrak{M}_1$ and by the Kaplansky density theorem any element $x \in M_1^\circ$ can be approximated in any von Neumann algebra topology by a bounded net of elements $x_\lambda^\circ =  x_\lambda - E_1(x_\lambda) \in \mathfrak{M}_1^\circ$ with $x_\lambda \in \mathfrak{M}_1$, $x_\lambda \longrightarrow x$. Thus $\mathfrak{M}_1 + \mathrm{span}(\Lambda^\circ(\mathfrak{M}_1^\circ,M_2^\circ)\setminus \mathfrak{M}_1^\circ)$ is also dense in $M$ in any von Neumann algebra topology so that the Kaplansky density theorem enables us to approximate each $x_k$ ($=x_k - E_{M_1}(x_k)$) by a net $x_{k,\lambda}$ in $\mathrm{span}(\Lambda^\circ(\mathfrak{M}_1^\circ,M_2^\circ)\setminus\mathfrak{M}_1^\circ)$; namely $\Vert x_{k,\lambda}\Vert_\infty \leq 2\Vert x_k\Vert_\infty$ and $x_{k,\lambda} \longrightarrow x_k$ $\sigma$-$*$-strongly. Then we have, for every $v \in A^u$, 
\begin{align*} 
&\Vert E_{M_1}(x_k^* v x_l)\Vert_\varphi 
\leq 
\Vert E_{M_1}(x_{k,\lambda}^* v x_{l,\lambda})\Vert_\varphi 
+ 
\Vert E_{M_1}(x_k^* v x_l - x_{k,\lambda}^* v x_{l,\lambda})\Vert_\varphi \\
&\leq 
\Vert E_{M_1}(x_{k,\lambda}^* v x_{l,\lambda})\Vert_\varphi 
+ \Vert(x_k - x_{k,\lambda})^* v x_l)\Vert_{\varphi\circ E_{M_1}} + \Vert x_{k,\lambda} v (x_l - x_{l,\lambda})\Vert_{\varphi\circ E_{M_1}} \\
&\leq 
\Vert E_{M_1}(x_{k,\lambda}^* v x_{l,\lambda})\Vert_\varphi 
+ \Vert \sigma_{i/2}^{\varphi\circ E_{M_1}}(x_l)\Vert_\infty \Vert x_k^* - x_{k,\lambda}^*\Vert_{\varphi\circ E_{M_1}} + 
2\Vert x_k\Vert_\infty \Vert x_l - x_{l,\lambda}\Vert_{\varphi\circ E_{M_1}}, 
\end{align*}
where we used, in the last line, that $x_l \in \mathcal{A}$ with \cite[Lemma VIII.3.18 (ii)]{Takesaki:Book2} and $v \in (M_1)_\varphi$. Let $\varepsilon > 0$ be arbitrary chosen. Then some $\lambda$ (being independent of $v$'s) satisfies that $\gamma^2 \leq \varepsilon + C \sum_{k,l=1}^m \Vert E_{M_1}(x_{k,\lambda}^* v x_{l,\lambda})\Vert_\varphi$ for all $v \in A^u$. Since any element in $\Lambda^\circ(\mathfrak{M}_1^\circ,M_2^\circ)\setminus\mathfrak{M}_1^\circ$ is written as $azb$ with $a,b \in \{1\}\cup \mathfrak{M}_1^\circ$, $z$ an alternating word in $\mathfrak{M}_1^\circ,M_2^\circ$ whose leftmost and rightmost letters are chosen from $M_2^\circ$, there are finitely many such words $a_j^{(i)}z_j^{(i)}b_j^{(i)}$, $i=1,2$, $j = 1,\dots,m'$, and positive constants $C_j > 0$, $j = 1,\dots,m'$, so that   
\begin{align*}
\gamma^2 &\leq \varepsilon + 
\sideset{}{_{j=1}^{m'}}\sum C_j \Vert E_{M_1}(a_j^{(1)}z_j^{(1)}b_j^{(1)} v a_j^{(2)}z_j^{(2)}b_j^{(2)})\Vert_\varphi \\
&= \varepsilon + 
\sideset{}{_{j=1}^{m'}}\sum C_j \Vert a_j^{(1)}E_{M_1}(z_j^{(1)} E_1(b_j^{(1)} v a_j^{(2)})z_j^{(2)})b_j^{(2)}\Vert_\varphi
\end{align*}
for all $v \in A^u$, where the equality comes from the free independence of $M_1, M_2$ and \eqref{Eq-2.2}. Applying the above estimate of $\gamma^2$ to the net $v=v_\lambda$ in our hypothesis we get $\gamma^2 \leq \varepsilon$ (at the limit in $\lambda$), a contradiction to $\gamma \gneqq 0$, since $\varepsilon$ is arbitrary. 
\end{proof} 

\begin{remark}\label{R-3.4} {\rm It is worth while to note that the inequality \eqref{Eq-3.6} is a general fact. Let $P \supseteq Q$ be $\sigma$-finite von Neumann algebras with a faithful normal conditional expectation $E_Q : P \rightarrow Q$ and $A$ be a von Neumann subalgebra of the centralizer $Q_\varphi$ with some faithful normal state $\varphi$. The middle part of discussion above shows that for each  projection $f \in A'\cap\langle P, Q\rangle$ satisfying both $f \leq 1-e_Q$ and $\Vert \widehat{E_Q}(f)\Vert_\infty < +\infty$ there are analytic (with respect to $\sigma^{\varphi\circ E_Q}$) elements $x_1,\dots,x_m \in P$ and a universal constant $C > 0$ such that 
\begin{equation*} 
\Vert \Lambda_{\varphi\circ E_Q\circ\hat{E}_Q}(f) \Vert_{\varphi\circ E_Q\circ\hat{E}_Q}^2 \leq 
C \sideset{}{_{k,l=1}^m}\sum \Vert E_Q(x_k^* v x_l)\Vert_\varphi \quad \text{for all $v \in A^u$}. 
\end{equation*}
}
\end{remark}   

\subsection{A result for controlling central sequences in amalgamated free products} Let us investigate central sequences in an amalgamated free product $(M,E) = (M_1,E_1)\star_N(M_2,E_2)$. The next result is an adaptation and/or an improvement of the methods of \cite[Proposition 3.5]{Ueda:AdvMath11} and \cite[Proposition 3.1]{Ueda:MRL} to  amalgamated free product von Neumann algebras. In this subsection we use the notations and facts summarized in \cite[\S\S2.2]{Ueda:AdvMath11}. 

\begin{proposition}\label{P-3.5} Suppose that there is a faithful normal state $\varphi$ on $M_1$ satisfying the following conditions{\rm:} 
\begin{itemize} 
\item[(a)] $\sigma_t^\varphi(N) = N$ for all $t \in \mathbb{R}$.  
\item[(b)] For every $n \in \mathbb{N}$ with $n \geq 2$ there are unitaries $u_k = u_k^{(n)}
,v_k = v_k^{(n)} \in (M_1)_\varphi$, $0 \leq k \leq n-1$, such that $E_1(u_{k_1}^* u_{k_2}) = E_N^\varphi(v_{k_1}^* v_{k_2}) = 0$ for all $0 \leq k_1 \neq k_2 \leq n-1$, where $E_N^\varphi$ denotes the unique $\varphi$-preserving conditional expectation from $M_1$ onto $N$, whose existence follows from {\rm(a)} and Takesaki's criterion. 
\end{itemize} 
Then, for any $x \in (M_1)_\varphi'\cap M^\omega$, any $y \in M_2^\circ$ and any sequence $(t_m)_m$ of real numbers we have 
\begin{align*}
\Vert E_2(y^* y)^{1/2}(x-(E_{M_1})^\omega(x))\Vert_{(\varphi\circ E_{M_1})^\omega} 
\leq 
\Vert yx - xz \Vert_{(\varphi\circ E_{M_1})^\omega}, 
\end{align*}  
with $z := \big[(\sigma_{t_m}^{\varphi\circ E_{M_1}}(y))_m\big] \in M^\omega$. 
\end{proposition}

Remark here that any bounded sequence $(\sigma_{t_m}^{\varphi\circ E_{M_1}}(x(m)))_m$ with arbitrary $(x(m))_m$ giving an element in $M^\omega$ gives again an element in $M^\omega$, as shown in the proof of \cite[Proposition 3.1]{Ueda:MRL}. A key fact behind this is that any modular action $\sigma^\psi$ satisfies $\psi\circ \sigma_{t}^\psi = \psi$ for all $t \in \mathbb{R}$. In particular, the element $z$ in the statement above makes sense.     

\begin{proof} Write $M_1^\triangledown := \mathrm{Ker}(E_N^\varphi)$. One can easily see, by using $x \in M_1 \mapsto$ $E_1(x) + (x - E_1(x)) \in N + M_1^\circ$ or $E_N^\varphi(x) + (x - E_N^\varphi(x)) \in N + M_1^\triangledown$, that $\mathrm{span}(\Lambda^\circ(M_1^\circ,M_2^\circ)\setminus M_1^\circ)$ coincides with the linear span of the following sets of words: 
\begin{equation*} 
\underbrace{M_1^\circ \cdots M_2^\circ}_{\text{alternating}}M_1^\triangledown, \quad \underbrace{M_1^\circ \cdots M_2^\circ}_{\text{alternating}}, \quad 
\underbrace{M_2^\circ \cdots M_2^\circ}_{\text{alternating}}M_1^\triangledown, \quad 
\underbrace{M_2^\circ \cdots M_2^\circ}_{\text{alternating}}.
\end{equation*} 
Define four closed subspaces $\mathcal{X}_1 := \left[\Lambda_{\varphi\circ E_{M_1}}(M_1^\circ \cdots M_2^\circ M_1^\triangledown)\right]$, $\mathcal{X}_2 := \left[\Lambda_{\varphi\circ E_{M_1}}(M_1^\circ \cdots M_2^\circ)\right]$, $\mathcal{X}_3 := \left[\Lambda_{\varphi\circ E_{M_1}}(M_2^\circ \cdots M_2^\circ M_1^\triangledown)\right]$, $\mathcal{X}_4 := \left[\Lambda_{\varphi\circ E_{M_1}}(M_2^\circ \cdots M_2^\circ)\right]$ in $\mathcal{H} := L^2(M,\varphi\circ E_{M_1})$, and clearly   
\begin{equation*} 
\mathcal{H} = \overline{\Lambda_{\varphi\circ E_{M_1}}(M_1)} \oplus \mathcal{X}_1 \oplus \mathcal{X}_2 \oplus \mathcal{X}_3 \oplus \mathcal{X}_4. 
\end{equation*} 
Denote by $P_i$, $i=1,2,3,4$, the projection from $\mathcal{H}$ onto $\mathcal{X}_i$. Remark that 
\begin{equation}\label{Eq-3.7} 
\left(I_{\mathcal{H}} - \sideset{}{_{i=1}^4}\sum P_i\right)\Lambda_{\varphi\circ E_{M_1}}(x) = \Lambda_{\varphi\circ E_{M_1}}(E_{M_1}(x)), \quad x \in M. 
\end{equation}

Let $n \in \mathbb{N}$ with $n \geq 2$ be fixed. Define unitary operators $S_k = S_k^{(n)}, T_k = T_k^{(n)}$ ($k = 0,\dots,n-1$) on $\mathcal{H}$ by 
\begin{equation*} 
S_k\Lambda_{\varphi\circ E_{M_1}}(x) := \Lambda_{\varphi\circ E_{M_1}}(u_k x u_k^*), \quad 
T_k\Lambda_{\varphi\circ E_{M_1}}(x) := \Lambda_{\varphi\circ E_{M_1}}(v_k x v_k^*), \quad x \in M, 
\end{equation*}
with $u_k = u_k^{(n)}, v_k = v_k^{(n)} \in (M_1)_\varphi \subseteq M_{\varphi\circ E_{M_1}}$ in our hypothesis. Here are simple claims. 
\begin{itemize} 
\item[(A)] $\{S_k\mathcal{X}_i\}_{k=0}^{n-1}$ is an orthogonal family of closed subspaces, $i=3,4$. 
\item[(B)] $\{T_k\mathcal{X}_2\}_{k=0}^{n-1}$ is an orthogonal family of closed subspaces.
\end{itemize}
The proofs of those are essentially same, but (A) is easier than (B). Thus we prove only (B) here and leave (A) to the reader. By using $x \mapsto E_i(x) + (x - E_i(x)) \in N + M_i^\circ$ ($i=1,2$) again and again we have
\begin{align*} 
(v_{k_2}&(M_1^\circ\cdots M_2^\circ)v_{k_2}^*)^* (v_{k_1}(M_1^\circ\cdots M_2^\circ)v_{k_1}^*) \\
&= v_{k_2} (M_2^\circ \cdots (M_1^\circ v_{k_2}^* v_{k_1} M_1^\circ)\cdots M_2^\circ) v_{k_1}^* \subseteq  v_{k_2} N v_{k_1}^* + v_{k_2}\mathrm{Ker}(E_{M_1})v_{k_1}^*. 
\end{align*}
The desired assertion immediately follows from that $v_k \in (M_1)_\varphi$; in fact, if $k_1 \neq k_2$, then 
\begin{align*} 
&\varphi\circ E_{M_1}(v_{k_2} N v_{k_1}^*) = \varphi(N v_{k_1}^* v_{k_2}) = \varphi(N E_N^\varphi(v_{k_1}^* v_{k_2})) = \{0\}, \\ 
&\varphi\circ E_{M_1}(v_{k_2}\mathrm{Ker}(E_{M_1})v_{k_1}^*) = \varphi(v_{k_2}E_{M_1}(\mathrm{Ker}(E_{M_1}))v_{k_1}^*) = \{0\}.
\end{align*} 

\medskip
Let us choose arbitrary $x \in (M_1)_\varphi'\cap M^\omega$ with representative $(x(m))_m$. For each $\varepsilon > 0$ and each $n \in \mathbb{N}$ with $n \geq 2$ one can choose a neighborhood $W=W_{\varepsilon,n}$ in $\beta(\mathbb{N})$ at $\omega$ so that 
\begin{equation*} 
\Vert\Lambda_{\varphi\circ E_{M_1}}(x(m)-u_k x(m) u_k^*)\Vert_{\mathcal{H}} < \varepsilon, \quad \Vert\Lambda_{\varphi\circ E_{M_1}}(x(m)-v_k x(m) v_k^*)\Vert_{\mathcal{H}} < \varepsilon
\end{equation*} 
for all $0 \leq k \leq n-1$ and $m \in W \cap \mathbb{N}$, where the $u_k = u_k^{(n)}, v_k = v_k^{(n)}$ are as above. For each $i=3,4$ and every $m \in W\cap\mathbb{N}$ we have, with the above $S_k = S_k^{(n)}$,  
\begin{align*} 
&\Vert P_i\Lambda_{\varphi\circ E_{M_1}}(x(m))\Vert_{\mathcal{H}}^2 \\
&= 
\frac{1}{n}\sideset{}{_{k=0}^{n-1}}\sum \Vert S_k P_i\Lambda_{\varphi\circ E_{M_1}}(x(m))\Vert_{\mathcal{H}}^2 \\
&\leq 
\frac{2}{n}\sideset{}{_{k=0}^{n-1}}\sum\Big\{ 
\Vert S_k P_i\Lambda_{\varphi\circ E_{M_1}}(x(m)) - 
S_k P_i S_k^*\Lambda_{\varphi\circ E_{M_1}}(x(m))\Vert_{\mathcal{H}}^2 + \Vert S_k P_i S_k^*\Lambda_{\varphi\circ E_{M_1}}(x(m))\Vert_{\mathcal{H}}^2\Big\} \\
&= 
\frac{2}{n}\sideset{}{_{k=1}^{n-1}}\sum
\Vert S_k P_i S_k^*\Lambda_{\varphi\circ E_{M_1}}(u_k x(m)u_k^* -x(m))\Vert_{\mathcal{H}}^2 + \frac{2}{n}\sideset{}{_{k=0}^{n-1}}\sum\Vert S_k P_i S_k^*\Lambda_{\varphi\circ E_{M_1}}(x(m))\Vert_{\mathcal{H}}^2 \\
&< 
2 \varepsilon^2 + \frac{2}{n}\sideset{}{_{k=0}^{n-1}}\sum\Vert S_k P_i S_{k}^*\Lambda_{\varphi\circ E_{M_1}}(x(m))\Vert_{\mathcal{H}}^2 \\
&\leq 
2 \varepsilon^2 + \frac{2}{n}\Vert\Lambda_{\varphi\circ E_{M_1}}(x(m))\Vert_{\mathcal{H}}^2 \quad \text{(by the claim (A))} \\
&\leq 
2 \varepsilon^2 + 2\Vert((x(m))_m\Vert_\infty^2/n. 
\end{align*} 
Similarly, using the claim (B) with $T_k^{(n)}$ instead of $S_k^{(n)}$ we have 
\begin{equation*} 
\Vert P_2\Lambda_{\varphi\circ E_{M_1}}(x(m))\Vert_{\mathcal{H}}^2 < 2\varepsilon^2 + 2\Vert((x(m))_m\Vert_\infty^2/n 
\end{equation*}
for every $m \in W \cap \mathbb{N}$. Since $n$ and $\varepsilon$ are arbitrary, for each $\delta > 0$ one can find a neighborhood $W_\delta$ in $\beta(\mathbb{N})$ at $\omega$ so that 
\begin{equation}\label{Eq-3.8}
\Vert (P_2+P_3+P_4)\Lambda_{\varphi\circ E_{M_1}}(x(m))\Vert_{\mathcal{H}} < \delta
\end{equation}
for all $m \in W_\delta \cap \mathbb{N}$. 

In the standard embedding $L^2(M^\omega,(\varphi\circ E_{M_1})^\omega) \hookrightarrow \mathcal{H}^\omega$ we have, by \eqref{Eq-3.7} and  \eqref{Eq-3.8},  
\begin{align*} 
&\Big\Vert\Lambda_{(\varphi\circ E_{M_1})^\omega}(y(x - (E_{M_1})^\omega(x))) - \big[(y P_1\Lambda_{\varphi\circ E_{M_1}}(x(m)))_m\big]\Big\Vert_{\mathcal{H}^\omega} \\
&= 
\lim_{m\rightarrow\omega} 
\big\Vert\Lambda_{\varphi\circ E_{M_1}}(y(x(m) - E_{M_1}(x(m)))) - y P_1\Lambda_{\varphi\circ E_{M_1}}(x(m))\big\Vert_{\mathcal{H}} \\
&= 
\lim_{m\rightarrow\omega} 
\big\Vert y(P_2 + P_3 + P_4)\Lambda_{\varphi\circ E_{M_1}}(x(m))\big\Vert_{\mathcal{H}} \\
&\leq 
\sup_{m \in W_\delta \cap \mathbb{N}} \big\Vert y(P_2 + P_3 + P_4)\Lambda_{\varphi\circ E_{M_1}}(x(m))\big\Vert_{\mathcal{H}} < 
\Vert y \Vert_\infty \delta, 
\end{align*}   
and hence 
\begin{equation}\label{Eq-3.9} 
\Lambda_{(\varphi\circ E_{M_1})^\omega}(y(x - (E_{M_1})^\omega(x))) = \big[(y P_1\Lambda_{\varphi\circ E_{M_1}}(x(m)))_m\big] 
\end{equation}
in $\mathcal{H}^\omega$, since $\delta$ is arbitrary. Trivially, in $\mathcal{H}^\omega$, 
\begin{align}\label{Eq-3.10} 
\Lambda_{(\varphi\circ E_{M_1})^\omega}(y(E_{M_1})^\omega(x) &- (E_{M_1})^\omega(x)z) \notag\\ &= \big[(\Lambda_{\varphi\circ E_{M_1}}(y E_{M_1}(x(m)) - E_{M_1}(x(m))\sigma_{t_m}^{\varphi\circ E_{M_1}}(y)))_m\big]. 
\end{align}
Set 
\begin{align*} 
y_\ell 
&:= 
\int_{-\infty}^{+\infty} \sigma_t^{\varphi\circ E_{M_1}}(y)\,\frac{e^{-t^2/\ell}\,dt}{\sqrt{\ell\pi}} \\
&=  
\int_{-\infty}^{+\infty} [D\varphi\circ E_{M_1}:D\chi\circ E]_t\,\sigma_t^{\chi\circ E}(y)\,[D\varphi\circ E_{M_1}:D\chi\circ E]_t^*\,\frac{e^{-t^2/\ell}\,dt}{\sqrt{\ell\pi}}
\end{align*} 
with a fixed faithful normal state $\chi$ on $N$. Clearly $y_\ell$ falls in the $\sigma$-weak (or $\sigma$-strong) closure of $\mathrm{span}(M_1 M_2^\circ M_1)$, since $[D\varphi\circ E_{M_1}:D\chi\circ E]_t = [D\varphi:D\chi\circ E_1]_t \in M_1$ by \cite[Corollary IX.4.22 (ii)]{Takesaki:Book2} and $\sigma_t^{\chi\circ E}(y) \in M_2^\circ$ by \eqref{Eq-2.1}. Set $z_\ell := \big[(\sigma_{t_m}^{\varphi\circ E_{M_1}}(y_\ell))_m\big] \in M^\omega$, which is well-defined as remarked just before the proof. Note that $\sigma_{-i/2}^{\varphi\circ E_{M_1}}(\sigma_{t_m}^{\varphi\circ E_{M_1}}(y_\ell)) = \sigma_{t_m}^{\varphi\circ E_{M_1}}(\sigma_{-i/2}^{\varphi\circ E_{M_1}}(y_\ell))$. For each $\ell$ we have, by \eqref{Eq-3.7}, \eqref{Eq-3.8} as before and by \cite[Lemma VIII.3.18 (ii)]{Takesaki:Book2}, 
\begin{align*} 
&\Big\Vert\Lambda_{(\varphi\circ E_{M_1})^\omega}((x - (E_{M_1})^\omega(x))z_\ell) - \big[(J \sigma_{-i/2}^{\varphi\circ E_{M_1}}(\sigma_{t_m}^{\varphi\circ E_{M_1}}(y_\ell))^* J P_1\Lambda_{\varphi\circ E_{M_1}}(x(m)))_m\big]\Big\Vert_{\mathcal{H}^\omega} \\
&= 
\lim_{m\rightarrow\omega} 
\big\Vert J \sigma_{-i/2}^{\varphi\circ E_{M_1}}(\sigma_{t_m}^{\varphi\circ E_{M_1}}(y_\ell))^* J \big(\Lambda_{\varphi\circ E_{M_1}}(x(m) - E_{M_1}(x(m))) - P_1\Lambda_{\varphi\circ E_{M_1}}(x(m))\big)\big\Vert_{\mathcal{H}} \\
&\leq 
\sup_{m \in W_\delta \cap \mathbb{N}} \big\Vert J \sigma_{-i/2}^{\varphi\circ E_{M_1}}(\sigma_{t_m}^{\varphi\circ E_{M_1}}(y_\ell))^* J(P_2 + P_3 + P_4)\Lambda_{\varphi\circ E_{M_1}}(x(m))\big\Vert_{\mathcal{H}} < 
\Vert \sigma_{-i/2}^{\varphi\circ E_{M_1}}(y_\ell) \Vert_\infty \delta 
\end{align*}   
with the modular conjugation $J$ of $M \curvearrowright \mathcal{H} = L^2(M,\varphi\circ E_{M_1})$. Hence, for each $\ell$,  
\begin{align}\label{Eq-3.11} 
\Lambda_{(\varphi\circ E_{M_1})^\omega}(&(x - (E_{M_1})^\omega(x))z_\ell) \notag\\ 
&= \big[(J \sigma_{-i/2}^{\varphi\circ E_{M_1}}(\sigma_{t_m}^{\varphi\circ E_{M_1}}(y_\ell))^* J P_1\Lambda_{\varphi\circ E_{M_1}}(x(m)))_m\big] 
\end{align}
in $\mathcal{H}^\omega$, since $\delta$ is arbitrary. Note that 
\begin{equation*}
y P_1\Lambda_{\varphi\circ E_{M_1}}(x_m))
\in \overline{\mathrm{span}\Lambda_{\varphi\circ E_{M_1}}(M_2^\circ M_1^\circ \cdots M_2^\circ M_1^\triangledown)}. 
\end{equation*}
On the other hand, 
\begin{align*}
\Lambda_{\varphi\circ E_{M_1}}(y E_{M_1}(x(m)) &- E_{M_1}(x(m))\sigma_{t_m}^{\varphi\circ E_{M_1}}(y)) \\
&\in \overline{\mathrm{span}\Lambda_{\varphi\circ E_{M_1}}(M_2^\circ)}\oplus \overline{\mathrm{span}\Lambda_{\varphi\circ E_{M_1}}(M_2^\circ M_1^\triangledown)}\\
&\quad\oplus\overline{\mathrm{span}\Lambda_{\varphi\circ E_{M_1}}(M_1^\circ M_2^\circ)}\oplus\overline{\mathrm{span}\Lambda_{\varphi\circ E_{M_1}}(M_1^\circ M_2^\circ M_1^\triangledown)}
\end{align*}
and
\begin{align*}
J \sigma_{-i/2}^{\varphi\circ E_{M_1}}&(\sigma_{t_m}^{\varphi\circ E_{M_1}}(y_\ell))^* J P_1\Lambda_{\varphi\circ E_{M_1}}(x(m)) \\
&\in 
\overline{\mathrm{span}\Lambda_{\varphi\circ E_{M_1}}(M_1^\circ\cdots M_2^\circ M_1^\triangledown \sigma_{t_m}^{\varphi\circ E_{M_1}}(y_\ell))} \\
&\subseteq 
\overline{\Lambda_{\varphi\circ E_{M_1}}(M_1)} \oplus \overline{\mathrm{span}\Lambda_{\varphi\circ E_{M_1}}(M_1^\circ M_2^\circ \cdots)}.  
\end{align*}
Here the last fact follows from \cite[Lemma VIII.3.18 (ii)]{Takesaki:Book2} and that $\sigma_{t_m}^{\varphi\circ E_{M_1}}(y_\ell)$ falls in the $\sigma$-strong closure of $\mathrm{span}(M_1 M_2^\circ M_1)$. Therefore, we see, by \eqref{Eq-3.9}--\eqref{Eq-3.11}, that $\Lambda_{(\varphi\circ E_{M_1})^\omega}(y(x - (E_{M_1})^\omega(x)))$ is orthogonal to both $\Lambda_{(\varphi\circ E_{M_1})^\omega}(y(E_{M_1})^\omega(x) - (E_{M_1})^\omega(x)z)$ and $\Lambda_{(\varphi\circ E_{M_1})^\omega}((x - (E_{M_1})^\omega(x))z_\ell)$. Finally, letting $\hat{x} := \big[(\sigma_{-t_m}^{\varphi\circ E_{M_1}}(x(m)))_m\big]$, $\hat{y} := \big[(\sigma_{-t_m}^{\varphi\circ E_{M_1}}(y))_m\big]$, both of which fall in $M^\omega$ as remarked just before the proof, we have 
\begin{align*} 
&\big(\Lambda_{(\varphi\circ E_{M_1})^\omega}((x - (E_{M_1})^\omega(x))z)|\Lambda_{(\varphi\circ E_{M_1})^\omega}(y(x - (E_{M_1})^\omega(x)))\big)_{(\varphi\circ E_{M_1})^\omega} \\
&= 
(\varphi\circ E_{M_1})^\omega((x-(E_{M_1})^\omega(x))^* y^* (x-(E_{M_1})^\omega(x))z) \\
&= 
(\varphi\circ E_{M_1})^\omega((\hat{x}-(E_{M_1})^\omega(\hat{x}))^* \hat{y}^* (\hat{x} - (E_{M_1})^\omega(\hat{x}))y) \\
&= 
\lim_{\ell \rightarrow \infty}(\varphi\circ E_{M_1})^\omega((\hat{x}-(E_{M_1})^\omega(\hat{x}))^* \hat{y}^* (\hat{x} - (E_{M_1})^\omega(\hat{x}))y_\ell) \\
&= 
\lim_{\ell \rightarrow \infty}(\varphi\circ E_{M_1})^\omega((x-(E_{M_1})^\omega(x))^* y^* (x - (E_{M_1})^\omega(x))z_\ell) \\
&= 
\lim_{\ell \rightarrow \infty} 
\big(\Lambda_{(\varphi\circ E_{M_1})^\omega}((x - (E_{M_1})^\omega(x))z_\ell|\Lambda_{(\varphi\circ E_{M_1})^\omega}(y(x - (E_{M_1})^\omega(x)))\big)_{(\varphi\circ E_{M_1})^\omega} = 0. 
\end{align*} 
Consequently we get $\Vert y(x-(E_{M_1})^\omega(x))\Vert_{(\varphi\circ E_{M_1})^\omega} \leq 
\Vert yx - xz \Vert_{(\varphi\circ E_{M_1})^\omega}$. We have, by \eqref{Eq-3.9},  
\begin{align*} 
&\Vert y(x - (E_{M_1})^\omega(x))\Vert_{(\varphi\circ E_{M_1})^\omega}^2 \\
&= 
\Big\Vert\big[(y P_1\Lambda_{\varphi\circ E_{M_1}}(x(m)))_m\big]\Big\Vert_{\mathcal{H}^\omega}^2 \\
&= 
\lim_{m\rightarrow\omega} 
(y P_1 \Lambda_{\varphi\circ E_{M_1}}(x(m))|y P_1 \Lambda_{\varphi\circ E_{M_1}}(x(m)))_{\varphi\circ E_{M_1}} \\
&= 
\lim_{m\rightarrow\omega} \Big\{ 
(E_2(y^*y) P_1 \Lambda_{\varphi\circ E_{M_1}}(x(m))| P_1 \Lambda_{\varphi\circ E_{M_1}}(x(m)))_{\varphi\circ E_{M_1}} \\
&\phantom{aaaaaaaaaa}+ (\underbrace{(y^* y - E_2(y^*y)) P_1 \Lambda_{\varphi\circ E_{M_1}}(x(m))}_{\text{in $\mathcal{X}_3$ orthogonal to $\mathcal{X}_1$}}| \underbrace{P_1 \Lambda_{\varphi\circ E_{M_1}}(x(m))}_{\text{in $\mathcal{X}_1$}})_{\varphi\circ E_{M_1}} \Big\} \\
&= 
\lim_{m\rightarrow\omega} 
(E_2(y^*y) P_1 \Lambda_{\varphi\circ E_{M_1}}(x(m))| P_1 \Lambda_{\varphi\circ E_{M_1}}(x(m)))_{\varphi\circ E_{M_1}} \\
&= 
\Big\Vert\big[(E_2(y^* y)^{1/2} P_1\Lambda_{\varphi\circ E_{M_1}}(x(m)))_m\big]\Big\Vert_{\mathcal{H}^\omega}^2.  
\end{align*} 
As in showing \eqref{Eq-3.9} one has 
\begin{equation*} 
\big[(E_2(y^* y)^{1/2} P_1\Lambda_{\varphi\circ E_{M_1}}(x(m)))_m\big] = i
\Lambda_{(\varphi\circ E_{M_1})^\omega}(E_2(y^* y)^{1/2}(x - (E_{M_1})^\omega(x))),  
\end{equation*}
and the proof is completed. 
\end{proof} 

\section{Some Consequences}  

We first formulate that $P$ is `non-trivial relative to $Q$' for a given inclusion of von Neumann algebras $P \supseteq Q$, and then provide some technical facts.  

\begin{definition}\label{D-4.1} A {\rm(}unital{\rm)} inclusion $P \supseteq Q$ of von Neumann algebras is said to be entirely non-trivial, if no non-zero direct summand of $Q$ is a direct summand of $P$. 
\end{definition}   

Let $P \supseteq Q$ be an inclusion of von Neumann algebras with a faithful normal conditional expectation $E_Q$. If $zP = Qz$ (as set) for some non-zero $z \in \mathcal{Z}(Q)^p$, then $Pz = Qz$ too by taking adjoints, and thus for $x \in P$ one has $zx = E_Q(zx) = zE_Q(x) = E_Q(x)z = E_Q(xz) = xz$, implying $z \in \mathcal{Z}(P)$. Hence $Qz$ is a direct summand of $P$. Therefore, $P \supseteq Q$ is entirely non-trivial if and only if $Pz \neq Qz$ or equivalently $zP \neq Qz$ for any non-zero projection $z \in \mathcal{Z}(Q)$, where $Pz$ and $zP$ denote the one-sided ideals of all $xz$ and $zx$, respectively, with $x \in P$. 

The next simple lemma, especially (3) there, will frequently be used later. 

\begin{lemma}\label{L-4.2} Let $P \supseteq Q$ be an inclusion of von Neumann algebras with a faithful normal conditional expectation $E_Q : P \rightarrow Q$. 

{\rm (1)} The following are equivalent{\rm:} 
\begin{itemize} 
\item[(i)] $P \supseteq Q$ is entirely non-trivial. 
\item[(ii)] $Pe \neq Qe$ or equivalently $eP \neq eQ$ for any non-zero projection $e \in Q$. 
\end{itemize} 

{\rm (2)} If $P \supseteq Q$ is entirely non-trivial and $f \in Q$ a projection with $c_f^Q = 1$, then $fPf \supseteq fQf$ is again entirely non-trivial. 

{\rm (3)} If $P \supseteq Q$ is entirely non-trivial, then there is a family $\{y_i\}_{i \in I}$ of elements in $\mathrm{Ker}(E_Q)$ so that $\sum_{i \in I} s(E_Q(y_i^* y_i)) = 1$, where $s(x)$ denotes the support projection of $x=x^*$.    
\end{lemma}
\begin{proof} (1) By the discussion above (i) is equivalent to $Pz \neq Qz$ or equivalently $zP \neq Qz$ for any non-zero $z \in \mathcal{Z}(Q)^p$. Thus (ii) $\Rightarrow$ (i) is trivial, and it suffices to show (i) $\Rightarrow$ (ii). Suppose that $Pe = Qe$ for some non-zero $e \in Q^p$. By a standard exhaustion argument based on the comparison theorem we can choose an orthogonal family $\{e_i\}_{i \in I}$ of projections in $Q$ such that $e_i \precsim e$ in $Q$ for all $i \in I$ and $c_e^Q = \sum_{i \in I} e_i$. Choose a partial isometry $v_i \in Q$ with $v_i^* v_i = e_i$ and $v_i v_i^* \leq e$, and then $P e_i = P v_i^* v_i \subseteq P e v_i = Q e v_i \subseteq Q e_i$, implying $Pe_i = Q e_i \subseteq Q$. For $x \in P$ one has $xc_e^Q = \sum_{i\in I} xe_i = \sum_{i\in I} E_Q(x)e_i = E_Q(x)c_e^Q$, and therefore $Pc_e^Q = Qc_e^Q$. 

(2) By (1) it suffices to prove that $ePf \neq eQf$ for any non-zero $e \in Q^p$ with $e \leq f$. As in (1) one can find an orthogonal family $\{f_i\}_{i\in I}$ of projections in $Q$ such that $f_i \precsim f$ in $Q$ for all $i \in I$ and $\sum_{i\in I} f_i = c_f^Q = 1$. On contrary, suppose that $ePf = eQf$ for some non-zero $e \in Q^p$ with $e \leq f$. Then one has $ePf_i = eQf_i$ in the same way as in (1). Hence, as in the above (1) one can justify, by using $E_Q$, the following computation: $eP = \sum_{i\in I} ePf_i = \sum_{i\in I}eQf_i = eQ$, a contradiction to the entire non-triviality of $P \supseteq Q$ thanks to (1). 

(3) Choose a maximal (with respect to set-inclusion) family $\{y_i\}_{i\in I}$ of elements in $\mathrm{Ker}(E_Q)$ so that $\{s(E_Q(y_i^* y_i))\}_{i \in I}$ is an orthogonal family of projections in $Q$. Suppose $\sum_{i\in I}s(E_Q(y_i^* y_i))\neq 1$. Set $e := 1 - \sum_{i\in I}s(E_Q(y_i^* y_i)) \in Q^p\setminus\{0\}$. Since $P \supseteq Q$ is entirely non-trivial, one has $Pe \neq Qe$ by (1), and hence can choose $x \in P$ with $xe \not\in Q$. Hence $xe - E_Q(xe) \neq 0$ and set $y := xe - E_Q(xe) \in \mathrm{Ker}(E_Q)$. Clearly, $ye = y$, and thus $E_Q(y^* y) = e E_Q(y^* y)e$, implying $s(E_Q(y^* y)) \leq e = 1 - \sum_{i\in I} s(E_Q(y_i^* y_i))$, a contradiction to the maximality of $\{y_i\}_{i\in I}$.        
\end{proof}  

Let $(M,E) = (M_1,E_1)\star_N(M_2,E_2)$ be an amalgamated free product throughout the rest of this section. 

\begin{theorem}\label{T-4.3} Assume that there is a faithful normal state $\varphi$ on $M_1$  such that one can find a {\rm(}possibly non-unital{\rm)} dense {\rm(}in any von Neumann algebra topology{\rm)} $*$-subalgebra\, $\mathfrak{M}_1$ of $M_1$ with $E_1(\mathfrak{M}_1) \subseteq \mathfrak{M}_1$ and a net $v_\lambda$ of unitaries in the centralizer $(M_1)_\varphi$ in such a way that $E_1(y^* v_\lambda x) \longrightarrow 0$ $\sigma$-strongly for all $x,y \in \mathfrak{M}_1$. Assume also that $M_2 \supseteq N$ is entirely non-trivial. Then we have{\rm:} 
\begin{itemize}
\item[(0)] $((M_1)_\varphi)' \cap M = ((M_1)_\varphi)' \cap M_1$.  
\item[(1)] $\mathcal{Z}(M) = \mathcal{Z}(M_1)\cap\mathcal{Z}(M_2)\cap\mathcal{Z}(N)$. 
\item[(2)] Let $\chi$ be an arbitrary faithful normal semifinite weight on $N$. Then, if a unitary $u$ in $M$ satisfies $\sigma_t^{\chi\circ E} = \mathrm{Ad}u$ for some $t \in \mathbb{R}$, then $u$ must fall in $N$. In particular, $T(M) = \{ t \in \mathbb{R}\,|\,\sigma_t^{\chi\circ E_1} = \mathrm{Ad}u = \sigma_t^{\chi\circ E_2}\ \text{for some $u \in N^u$}\}$. 
\item[(3)] $M$ is semifinite if and only if there is a faithful normal semifinite trace $\mathrm{Tr}_N$ such that both $\mathrm{Tr}_N\circ E_1$ and $\mathrm{Tr}_N\circ E_2$ are traces. 
\item[(4)] $\mathcal{Z}(\widetilde{M}) = \mathcal{Z}(\widetilde{M}_1)\cap\mathcal{Z}(\widetilde{M}_2)\cap\mathcal{Z}(\widetilde{N})$. 
\end{itemize} 
\end{theorem} 
\begin{proof} (0) is nothing but what Proposition \ref{P-3.3} says. 

(1) Let $x \in \mathcal{Z}(M)$ be arbitrary, and then $x$ must be in $M_1$ by (0). For any $y \in M_2^\circ$ one has $y(x-E_1(x)) + yE_1(x)=yx=xy=E_1(x)y + (x-E_1(x))y$, and thus $\{y E_1(x), E_1(x)y\}$, $y(x-E_1(x))$ and $(x-E_1(x))y$ are orthogonal with respect to $E$ due to the free independence between $M_1$ and $M_2$. Thus $y(x-E_1(x))=0$ so that (by looking at the $E$-value of the product of its adjoint and itself) we get $(x-E_1(x))^* E_2(y^* y) (x-E_1(x)) = 0$. Therefore, $E_2(y^* y)(x-E_1(x))= 0$ for all $y \in M_2^\circ$. By taking its adjoint one can easily see that $(x-E_1(x))^*\!\upharpoonright_{\mathrm{ran}(E_2(y^* y))} \equiv 0$ so that $(x-E_1(x))^* s(E_2(y^* y)) = 0$ for all $y \in M_2^\circ$. By Lemma \ref{L-4.2} (3) one can find a family $\{y_i\}_{i\in I}$ of elements in $M_2^\circ$ so that $\sum_{i\in I}s(E_2(y_i^* y_i)) = 1$, which implies $x = E_1(x) \in N$. The desired assertion is now immediate. 

(2) One has $\sigma_t^{\varphi\circ E_{M_1}} = \mathrm{Ad}([D\varphi:D\chi\circ E_1]_t\,u)$ by Connes's Radon--Nikodym cocycle theorem and \cite[Corollary IX.4.20]{Takesaki:Book2}. Since $(M_1)_\varphi \subseteq M_{\varphi\circ E_{M_1}}$, we have $[D\varphi:D\chi\circ E_1]_t\,u \in M_1$ by (0). In particular, $u \in M_1$, since $[D\varphi:D\chi\circ E_1]_t \in M_1^u$. For $y \in M_2^\circ$ we have $\sigma_t^{\chi\circ E}(y)(u-E_1(u)) + \sigma_t^{\chi\circ E}(y)E_1(u) = \sigma_t^{\chi\circ E}(y)u = uy = E_1(u)y + (u-E_1(u))y$, and as in (1) we get $(u-E_1(u))y = 0$, since $\sigma_t^{\chi\circ E}(y) = \sigma_t^{\chi\circ E_2}(y) \in M_2^\circ$ by \eqref{Eq-2.1}. The same argument as in (1) again shows $u = E_1(u) \in N$. The T-set computation is straightforward. 

(3) $M$ is semifinite if and only if there is a $1$-parameter unitary group $u(t)$ in $M$ so that $\sigma_t^{\chi\circ E}=\mathrm{Ad}u(t)$, $t \in \mathbb{R}$, for a fixed faithful normal state $\chi$ on $N$. See \cite[Theorem VIII.3.14]{Takesaki:Book2}. Then $u(t) \in N$ by (2). By Stone's theorem $u(t) = H^{it}$ with some positive non-singular, self-adjoint $H$ affiliated with $N$. Since $\sigma_t^\chi(u(t)) = \sigma_t^{\chi\circ E}(u(t))=u(t)$, $H$ must indeed be affiliated with the centralizer $N_\chi$. Hence, by \cite[Lemma VIII.2.8]{Takesaki:Book2} we can construct a faithful normal semifinite weight $\chi_{H^{-1}}$ on $N$, and by the construction we observe that $\chi_{H^{-1}}\circ E = (\chi\circ E)_{H^{-1}}$. Moreover, by \cite[Lemma VIII.2.11]{Takesaki:Book2} we have $\sigma_t^{\chi_{H^{-1}}\circ E} = H^{-it}\sigma_t^{\chi\circ E}(-)H^{it} = \mathrm{id}$. Hence the $\chi_{H^{-1}}$ is a desired faithful normal semifinite trace on $N$.        

(4) By (0) together with the same argument as in \cite[Corollary 4]{Ueda:MathScand01} we observe that $((M_1)_\varphi)'\cap(M\rtimes_{\sigma^{\varphi\circ E_{M_1}}}\mathbb{R}) = ((M_1)_\varphi)'\cap(M_1\rtimes_{\sigma^\varphi}\mathbb{R})$, where $(M_1)_\varphi \subset M_1 \subseteq M \hookrightarrow M\rtimes_{\sigma^{\varphi\circ E_{M_1}}}\mathbb{R}$ canonically as in \S2. It follows that $(\widetilde{M}_1)'\cap\widetilde{M} = \mathcal{Z}(\widetilde{M}_1)$, where we need Connes's Radon--Nikodym cocycle theorem together with \cite[Theorem X.1.7]{Takesaki:Book2}. Choose an arbitrary $x \in \mathcal{Z}(\widetilde{M})$. Then $x$ must fall in $\mathcal{Z}(\widetilde{M}_1) \subseteq \widetilde{M}_1$. For $y \in M_2^\circ \subset \widetilde{M}_2^\circ$ one has $y(x-\widetilde{E}(x)) + y\widetilde{E}(x) = yx = xy = \widetilde{E}(x)y + (x-\widetilde{E}(x))y$, and thus $y(x-\widetilde{E}(x)) = 0$ since $\widetilde{M}_1$, $\widetilde{M}_2$ are freely independent with respect to $\widetilde{E}$ as remarked in \S2. In particular, we get $E_2(y^* y) (x -\widetilde{E}(x)) = 0$ for all $y \in M_2^\circ$ as in (1). Therefore, using Lemma \ref{L-4.2} (3) as in (1) once again we can prove $x = \widetilde{E}(x) \in \widetilde{N}$. Hence we are done.   
\end{proof}

Let us illustrate how the above theorem is useful by giving next two corollaries. The first corollary shows that Proposition \ref{P-3.1} is useful to confirm the necessary hypothesis of the theorem. The second one does that the theorem is still applicable beyond the case where $N$ is semifinite. Remark that the first one can be viewed as a simultaneous generalization of both \cite[Theorem 3.4]{Ueda:AdvMath11} and \cite[\S4]{Ueda:PacificJMath99}. 

\begin{corollary}\label{C-4.4} Assume that $M_1$ is diffuse, $N$ of type I and $M_2 \supseteq N$ entirely non-trivial. Let $z \in \mathcal{Z}(N)$ be the unique projection so that $Nz$ is diffuse and $Nz^\perp$ atomic, and assume further that $M_1 c_{z}^{M_1}$ has no type I direct summand when $z \neq 0$ {\rm(}i.e., this last assumption is fulfilled if $M_1$ has no type I direct summand{\rm)}. Then all the assertions of Theorem \ref{T-4.3} holds with a certain faithful normal state $\varphi$ on $M_1$. 
\end{corollary} 
\begin{proof} Let us fix a faithful normal semifinite trace $\mathrm{Tr}_N$ on $N$. Write $c := c_z^{M_1}$ for simplicity. Clearly $\sigma_t^{\mathrm{Tr}_N\circ E_1}(c) = c$ for all $t \in \mathbb{R}$, and thus Takesaki's criterion shows that there is a $\mathrm{Tr}_N\circ E_1$-preserving unique conditional expectation $E_L : M_1 \rightarrow L := N \vee \{c\}'' = Nc \oplus Nc^\perp$ ($\supseteq N$). In particular, one observes that $E_1 \circ E_L = E_1$ holds.  
As in the proof of \cite[Theorem 3.4]{Ueda:AdvMath11} one can choose a faithful normal state $\varphi$ on $M_1$ such that $(M_1 c)_{\varphi\!\upharpoonright_{M_1 c}}$ has no type I direct summand and $(M_1 c^\perp)_{\varphi\!\upharpoonright_{M_1 c\perp}}$ is just only diffuse. Then it is clear that $(M_1 c)_{\varphi\!\upharpoonright_{M_1 c}} \not\preceq_{M_1 c} Nc$ with $E_L\!\upharpoonright_{M_1 c}$ and $\mathrm{Tr}_N\circ E_1\!\upharpoonright_{Nc}$ and that $(M_1 c^\perp)_{\varphi\!\upharpoonright_{M_1 c^\perp}} \not\preceq_{M_1 c^\perp} Nc^\perp$ with $E_L\!\upharpoonright_{M_1 c}$ and $\mathrm{Tr}_N\circ E_1\!\upharpoonright_{Nc^\perp}$, since $Nc^\perp = (Nz^\perp)c^\perp$ is a reduced von Neumann algebra of the atomic part $Nz^\perp$. Therefore, by the equivalent condition (i) in Proposition \ref{P-3.1} there are two nets $v_\lambda^{(1)}$ and $v_\lambda^{(2)}$ of unitaries in $(M_1 c)_{\varphi\!\upharpoonright_{M_1 c}}$ and $(M_1 c^\perp)_{\varphi\!\upharpoonright_{M_1 c^\perp}}$, respectively, so that $E_L(y_1^* v_\lambda^{(1)} x_1) \longrightarrow 0$ and $E_L(y_2^* v_\lambda^{(2)} x_2) \longrightarrow 0$ $\sigma$-strongly for all $x_1,y_1 \in \bigcup\{M_1 p\,|\, p \in (Nc)^p; \mathrm{Tr}_N\circ E_1(p) < +\infty\}$ and all $x_2,y_2 \in \bigcup\{M_1 p\,|\, p \in (Nc^\perp)^p; \mathrm{Tr}_N\circ E_1(p) < +\infty\}$. Remark that $E_L = (E_L\!\upharpoonright_{M_1 c})\oplus(E_L\!\upharpoonright_{M_1 c^\perp})$ in $M_1 = M_1 c \oplus M_1 c^\perp$ and that $\mathrm{Tr}_N(p) < +\infty$ implies both $\mathrm{Tr}_N\circ E_1(pc)<+\infty$ and $\mathrm{Tr}_N\circ E_1(pc^\perp) < +\infty$ for $p \in N^p$. Thus, letting $v_\lambda := v_\lambda^{(1)}\oplus v_\lambda^{(2)} \in (M_1 c)_{\varphi\!\upharpoonright_{M_1 c}}\oplus(M_1 c^\perp)_{\varphi\!\upharpoonright_{M_1 c^\perp}} = (M_1)_\varphi$ one has, for all $x,y \in \bigcup\{M_1 p\,|\,p \in N^p; \mathrm{Tr}_N(p) < +\infty\}$, $E_L(y^* v_\lambda x) \longrightarrow 0$ $\sigma$-strongly and hence $E_1(y^* v_\lambda x) = E_1(E_L(y^* v_\lambda x)) \longrightarrow 0$ $\sigma$-strongly. Hence we can apply Theorem \ref{T-4.3} with the above $\varphi$ and $\mathfrak{M}_1 := \bigcup\{pM_1 p\,|\,p \in N^p; \mathrm{Tr}_N(p)<+\infty\}$. Note here that $\mathfrak{M}_1$ is indeed a $*$-algebra thanks to the Kaplansky formula \cite[Theorem 6.1.7]{KadisonRingrose:Book2} and dense in any von Neumann algebra topology due to the semifiniteness of $\mathrm{Tr}_N$. 
\end{proof} 
  
\begin{corollary}\label{C-4.5} Assume that $(M_1,E_1)$ is one of the following{\rm:} {\rm (i)} $M_1 = N\rtimes_\alpha G$ and $E_1$ is the canonical conditional expectation from $M_1 = N\rtimes_\alpha G$ onto $N$, where $\alpha : G \curvearrowright N$ is an infinite discrete group action preserving a faithful normal state $\psi$ on $N$. {\rm(ii)} $M_1 = Q\bar{\otimes}N$ and $E_1 = \psi\bar{\otimes}\mathrm{id}_N$, where $Q$ is a diffuse von Neumann algebra with a faithful normal state $\psi$. Assume also that $M_2 \supseteq N$ is entirely non-trivial. Then all the assertions of Theorem \ref{T-4.3} holds with $\varphi = \psi\circ E_1$ in {\rm(i)} and with $\varphi = \varphi_0\bar{\otimes}\chi$ in {\rm(ii)}, where $Q_{\varphi_0}$ is diffuse {\rm(}such a state $\varphi_0$ certainly exists{\rm)} and $\chi$ arbitrary.    
\end{corollary}
\begin{proof} Case (i): Since $\psi$ is invariant under the action $\alpha$, the restriction $(\psi\bar{\otimes}\mathrm{id}_{B(\ell^2(G))})\!\upharpoonright_{N\rtimes_\alpha G}$ gives a faithful normal conditional expectation from $E_\psi : M_1 = N\rtimes_\alpha G \rightarrow L(G) = \mathbb{C}1\rtimes G$, and it is plain to see that $\psi\circ E_1 := \tau_G\circ E_\psi$ with the canonical tracial state $\tau_G$ on $L(G)$. Clearly $L(G) = \mathbb{C}1\rtimes G$ sits inside $(N\rtimes_\alpha G)_{\psi\circ E_1}$ and is diffuse (see e.g.~\cite[Proposition 5.1]{Dykema:DukeMathJ93}). With $\varphi := \psi\circ E_1 = \tau_G\circ E_\psi$ and $\mathfrak{M}_1 := \mathrm{span}\{ x\lambda_g\,|\, x \in N, g \in G\}$ one can choose a net $v_\lambda$ from $L(G) = \mathbb{C}\rtimes G$ as in Theorem \ref{T-4.3}, since $L(G)$ is diffuse and $E_1\!\upharpoonright_{L(G)=\mathbb{C}1\rtimes G} = \tau_G(-)1$. 

Case (ii): As in the proof of \cite[Theorem 2.4]{Ueda:AdvMath11} one can choose a faithful normal state $\varphi_0$ on $Q$ in such a way that the centralizer $Q_{\varphi_0}$ is diffuse. Set $\varphi := \varphi_0\bar{\otimes}\chi$ with a faithful normal state $\chi$ on $N$ and $\mathfrak{M}_1 := Q\odot N = \mathrm{span}\{ x\otimes y\,|\, x \in Q, y \in N\}$. Then one can choose a net $v_\lambda$ from $Q_{\varphi_0}\bar{\otimes}\mathbb{C}1$ as in Theorem \ref{T-4.3}, since $Q_{\varphi_0}$ is diffuse.    
\end{proof} 

The next lemma seems well-known, but we do give it for the reader's convenience as a reference for the discussions below.  

\begin{lemma}\label{L-4.6} Let $(P,F) = (P_1,F_1)\star_Q(P_2,F_2)$ be an amalgamated free product. If a projection $f \in Q$ has $c_f^Q = 1$, then $(fPf,F\!\upharpoonright_{fPf}) = (fP_1 f,F_1\!\upharpoonright_{fP_1 f})\star_{fQf}(fP_2 f,F_2\!\upharpoonright_{fP_2 f})$ holds canonically. 
\end{lemma} 
\begin{proof} Clearly $fP_1 f$ and $fP_2 f$ are freely independent with respect to $F\!\upharpoonright_{fPf}$, and hence it suffices to see that those generate $fPf$ as von Neumann algebra. As in the proof of Lemma \ref{L-4.2} one can find partial isometries $\{v_i\}_{i\in I}$ in $Q$ such that $\sum_{i\in I}v_i^* v_i = c_f^Q = 1$ and $v_i v_i^* \leq f$ for all $i \in I$. For any alternating word $x = x_1 \cdots x_n \in \Lambda^\circ(P_1^\circ,P_2^\circ)$ one has $fxf = \sum_{i_1,\dots,i_{n-1} \in I} (f x_1 v_{i_1}^*) (v_{i_1} x_2 v_{i_2}^*) \cdots (v_{i_{n-1}} x_n f)$ $\sigma$-strongly, which falls in the $\sigma$-strong closure of the linear span of $\Lambda^\circ((fP_1 f)^\circ, (fP_2 f)^\circ))$. Since $P$ is the $\sigma$-strong closure of $Q + \mathrm{span}\Lambda^\circ(P_1^\circ, P_2^\circ)$,  the assertion is immediate. 
\end{proof}    

\begin{lemma}\label{L-4.7} Let $P \supseteq Q$ be an inclusion of $\sigma$-finite von Neumann algebras with a faithful normal conditional expectation $E_Q : P \rightarrow Q$, and assume that $Q$ is commutative. 

{\rm(1)} If $P$ has no type I direct summand and a faithful normal semifinite trace $\mathrm{Tr}_P$ on $P$ with $\mathrm{Tr}_P\circ E_Q = \mathrm{Tr}_P$, then there is a faithful normal state $\chi$ on $Q$ so that for each $n \in \mathbb{N}$ with $n \geq 2$ one can find a unitary $u_n \in P_{\chi\circ E_Q}$ in such a way that $E_Q(u_n^k) = 0$ for all $1\leq k \leq n-1$, i.e., $E_Q(u_n^{k_1}{}^*\,u_n^{k_2}) = 0$ for all $0 \leq k_1 \neq k_2 \leq n-1$. 

{\rm(2)} If $P$ is diffuse and $Q$ is atomic, then there is a faithful normal state $\varphi$ on $P$ such that 
\begin{itemize} 
\item[(a)] the centralizer $P_\varphi$ contains $Q$, 
\item[(b)] there are two unitaries $u, v \in P_\varphi$ so that $E_Q(u^k) = E_Q^\varphi(v^k) = 0$ as long as $k\neq 0$, i.e., $E_Q(u^{k_1}{}^*\,u^{k_2}) = E_Q^\varphi(v^{k_1}{}^*\,v^{k_2}) = 0$ for all $k_1 \neq k_2$. Here $E_Q^\varphi$ denotes the unique $\varphi$-preserving conditional expectation from $P$ onto $Q$ whose existence follows from {\rm(a)} and Takesaki's criterion. 
\end{itemize}

{\rm(3)} Let $z \in \mathcal{Z}(P)$ be the central support projection of the type I direct summand of $P$. Assume that $P$ is diffuse and $Qz$ atomic. Then there is a faithful normal state $\varphi$ on the continuous core $\widetilde{P}$ of $P$ such that 
\begin{itemize}
\item[(a)] the centralizer $(\widetilde{P})_\varphi$ contains $\widetilde{Q}$, where $\widetilde{Q} = Q \rtimes_{\sigma^\chi}\mathbb{R} \hookrightarrow \widetilde{P} = P\rtimes_{\sigma^{\chi\circ E_Q}}\mathbb{R}$ with a faithful normal state or semifinite weight $\chi$ on $Q$, 
\item[(b)] for each $n \in \mathbb{N}$ with $n \geq 2$ one can find a unitary $u_n \in (\widetilde{P})_\varphi$ in such a way that $\widetilde{E}_Q(u_n^k) = E_{\widetilde{Q}}^\varphi(u_n^k) = 0$ for all $1 \leq k \leq n-1$, i.e., $\widetilde{E}_Q(u_n^{k_1}{}^*\,u_n^{k_2}) = E_{\widetilde{Q}}^\varphi(v_n^{k_1}{}^*\,v_n^{k_2}) = 0$ for all $0 \leq k_1 \neq k_2 \leq n-1$. Here $\widetilde{E}_Q = (E_Q\bar{\otimes}\mathrm{id}_{B(L^2(\mathbb{R}))})\!\upharpoonright_{\widetilde{P}}$, and $E_{\widetilde{Q}}^\varphi$ denotes the unique $\varphi$-preserving conditional expectation from $\widetilde{P}$ onto $\widetilde{Q}$ as in {\rm(2)}.   
\end{itemize}  
The same assertion also holds for $P \supseteq Q$ with $E_Q$ themselves, if it is further assumed that $P$ is semifinite and $E_Q$ preserves a faithful normal semifinite trace $\mathrm{Tr}_P$ on $P$. 
\end{lemma} 
\begin{proof} (1) By assumption $\mathrm{Tr}_P\!\upharpoonright_Q$ is semifinite, and thus one can choose an orthogonal sequence $\{q_m\}_m$ of projections in $Q$ with $\mathrm{Tr}_P(q_m) < +\infty$ and $\sum_{m\in\mathbb{N}} q_m = 1$. Consider the faithful normal state $\chi := \sum_{m\in\mathbb{N}} \frac{1}{2^m \mathrm{Tr}_P(q_m)} \mathrm{Tr}_P\!\upharpoonright_{Qq_m}$ on $Q$. (Remark here that $Q$ is commutative.) Clearly the centralizer $P_{\chi\circ E_Q}$ contains $\sum^\oplus_{m\in\mathbb{N}} q_m P q_m$ $\big(\supseteq \sum^\oplus_{m\in\mathbb{N}} Qq_m = Q\big)$ so that $P_{\chi\circ E_Q}$ must be of type II$_1$. Choose a MASA $\mathfrak{A}$ in $P_{\chi\circ E_Q}$ that contains $Q$. By \cite[Corollary 3.16]{Kadison:AmerJMath84}, for each $n \in \mathbb{N}$ with $n \geq 2$ there are $n$ orthogonal $e_0,\dots,e_{n-1} \in \mathfrak{A}^p$, all of which are equivalent in $P_{\chi\circ E_Q}$, and $\sum_{i=0}^{n-1} e_i = 1$. Then one can construct a unitary $u_n \in P_{\chi\circ E_Q}$ such that $u_n e_0 = e_1 u_n, u_n e_1 = e_2 u_n, \dots, u_n e_{n-1} = e_0 u_n$. Let $E_{\mathfrak{A}} : P \rightarrow \mathfrak{A}$ be the $\chi\circ E_Q$-preserving conditional expectation (whose existence follows from Takesaki's criterion), and clearly $E_Q\circ E_{\mathfrak{A}} = E_Q$. Then, for every $1 \leq k \leq n-1$ one has $E_{\mathfrak{A}}(u_n^k) = 0$ so that $E_Q(u_n^k)= E_Q(E_{\mathfrak{A}}(u_n^k)) = 0$. 

(2) Write $Q = \sum_{m \in \mathbb{N}}^\oplus \mathbb{C}q_m$. Clearly $E_Q$ factors as  
$P \overset{E_{Q'\cap P}}{\longrightarrow} Q'\cap P \overset{\Psi}{\longrightarrow} Q$, where $Q'\cap P = \sum_{m\in\mathbb{N}}^\oplus q_m P q_m$ and $E_{Q'\cap P}(x) = \sum_{m \in \mathbb{N}} q_m x q_m$ for $x \in P$. Moreover $\Psi$ is of the form $\Psi(\sum_{m\in\mathbb{N}} x_m) = \sum_{m \in \mathbb{N}} \psi_m(x_m)q_m$ for $x_m \in q_m P q_m$ with faithful normal states $\psi_m$ on $q_m P q_m$. Since $P$ is diffuse, so are all $q_m P q_m$; hence by the proof of \cite[Theorem 3.4]{Ueda:AdvMath11} there are faithful normal states $\varphi_m$ on $q_m P q_m$ with $(q_m P q_m)_{\varphi_m}$ diffuse for all $m$. Define $\Phi(\sum_{m\in\mathbb{N}} x_m) = \sum_{m\in\mathbb{N}}\varphi_m(x_m)q_m$ for $x_m \in q_m P q_m$, giving a faithful normal conditional expectation from $Q'\cap P$ onto $Q$. Set $\varphi := \chi\circ\Phi\circ E_{Q'\cap P}$, a faithful normal state on $P$, with a faithful normal state $\chi$ on $Q$. Then $Q' \cap P_{\varphi} = \sum_{m\in\mathbb{N}}^\oplus (q_m P q_m)_{\varphi_m}$, a direct sum of diffuse von Neumann algebras. One can choose, for each $m$, unitaries $u_m, v_m \in (q_m P q_m)_{\varphi_m}$ so that $\varphi_m(u_m^k) = \psi_m(v_m^k) = 0$ as long as $k \neq 0$. (See the proof of \cite[Theorem 3.7]{Ueda:AdvMath11}.) Then $u := \sum_{m\in\mathbb{N}} u_m$, $v := \sum_{m\in\mathbb{N}} v_m$ are unitaries in $Q' \cap P_\varphi$, and moreover $E_Q(u^k) = \Psi(u^k) = 0$ and $E_Q^\varphi(v^k) = \Phi(v^k) = 0$ as long as $k \neq 0$. 

(3) Consider $P = Pz \oplus Pz^\perp \supseteq R := Q\vee\{z\}'' = Qz\oplus Qz^\perp \supseteq Q$. Let $\chi$ be an arbitrary faithful normal state on $Q$. As in the proof of Corollary \ref{C-4.4} one can show that there is a unique faithful normal conditional expectation $E_R : P \rightarrow R$ with $E_Q\circ E_R = E_Q$. Then we have 
\begin{equation*}
\widetilde{P} = P\rtimes_{\sigma^{\chi\circ E_Q}}\mathbb{R} \overset{\widetilde{E}_R}{\supseteq} \widetilde{R} = R\rtimes_{\sigma^{\chi\circ(E_Q\!\upharpoonright_{R})}}\mathbb{R} \overset{\widetilde{E_Q\!\upharpoonright_R}}{\supseteq} \widetilde{Q} = Q \rtimes_{\sigma^\chi}\mathbb{R},  
\end{equation*}  
where $\widetilde{E}_R = (E_R\bar{\otimes}\mathrm{id}_{B(L^2(\mathbb{R}))})\!\upharpoonright_{\widetilde{P}}$ and 
$\widetilde{{E_Q}\!\upharpoonright_R} = 
(({E_Q}\!\upharpoonright_R)\bar{\otimes}\mathrm{id}_{B(L^2(\mathbb{R}))})\!\upharpoonright_{\widetilde{R}} = 
{\widetilde{E}_Q}\!\upharpoonright_{\widetilde{R}}$. Since $E_R = ({E_R}\!\upharpoonright_{Pz})\oplus({E_R}\!\upharpoonright_{Pz^\perp})$ in $P = Pz\oplus Pz^\perp$, we have, by \cite[Theorem X.1.7 (ii)]{Takesaki:Book2}, 
\begin{equation*} 
\Big(\widetilde{P} \overset{\widetilde{E}_R}{\supseteq} \widetilde{R}\Big) \cong 
\Big(\widetilde{Pz} \overset{\widetilde{{E_R}\!\upharpoonright_{Pz}}}{\supseteq} \widetilde{Qz}\Big)\oplus \Big(\widetilde{Pz^\perp} \overset{\widetilde{{E_R}\!\upharpoonright_{Pz^\perp}}}{\supseteq} \widetilde{Qz^\perp}\Big),   
\end{equation*} 
where the continuous cores and the conditional expectations in the right-hand side are defined similarly as above. Since $\widetilde{Pz^\perp}$ has no type I direct summand by the assumption here and \cite[Theorem XII.1.1]{Takesaki:Book2} and since $\widetilde{E_R\!\upharpoonright_{Pz^\perp}}$ preserves the canonical trace on $\widetilde{Pz^\perp}$ see e.g.~\cite[\S4]{Longo:CMP89}, we can apply (1) to the second $\Big(\widetilde{Pz^\perp} \supseteq \widetilde{Qz^\perp}\Big)$ with $\widetilde{E_R\!\upharpoonright_{Pz^\perp}}$ directly, and get a faithful normal state $\varphi_{z^\perp}$ on $\widetilde{Pz^\perp}$ with $\varphi_{z^\perp}\circ(\widetilde{E_R\!\upharpoonright_{Pz^\perp}}) = \varphi_{z^\perp}$ such that for each $n \in \mathbb{N}$ with $n \geq 2$ one can find a unitary $u_{z^\perp,n} \in (\widetilde{Pz})_{\varphi_{z^\perp}}$ in such a way that $\widetilde{E}_R(u_{z^\perp,n}^k) = (\widetilde{E_R\!\upharpoonright_{Pz^\perp}})(u_{z^\perp,n}^k) = 0$ for all $1 \leq k \leq n-1$. Write $Qz = \sum_{m\in\mathbb{N}}^\oplus \mathbb{C}e_m$, and $E_R\!\upharpoonright_{Pz}$ factors as $Pz \overset{E_{(Qz)'\cap Pz}}{\longrightarrow} (Qz)'\cap Pz \overset{\Psi}{\longrightarrow} Qz$, where $(Qz)'\cap Pz = \sum_{m\in\mathbb{N}}^\oplus e_m (Pz) e_m$ and $E_{(Qz)'\cap Pz}(x) = \sum_{m \in \mathbb{N}} e_m x e_m$ for $x \in Pz$. Moreover, $\Psi$ is of the form $\Psi(\sum_{m\in\mathbb{N}} x_m) = \sum_{m \in \mathbb{N}} \psi_m(x_m)e_m$ for $x_m \in e_m(Pz)e_m$ with faithful normal states $\psi_m$ on $e_m(Pz)e_m$. By the assumption here $Pz$ is diffuse and of type I, and thus so are the $e_m(Pz)e_m$; hence the centers of those must be diffuse, and so are all the $(e_m(Pz)e_m)_{\psi_m}$. In the same way as in (2), one can find a unitary $u_z \in ((Qz)'\cap Pz)_{\chi_z \circ \Psi}$ with `any' faithful normal state $\chi_z$ on $Qz$ in such a way that $\Psi(u_z^k) = 0$ for all $k \neq 0$. Denote by $\lambda(t)$ the generators of $\mathbb{C}\rtimes\mathbb{R}$ in $\widetilde{Pz} = (Pz)\rtimes_{\sigma^{\chi_z\circ(E_R\!\upharpoonright_{Pz})}}\mathbb{R}$ ($\hookleftarrow (Qz)\rtimes_{\sigma^{\chi_z}}\mathbb{R} = \widetilde{Qz}$ canonically), and set $\varphi_z := \tau\circ(\widetilde{E_R\!\upharpoonright_{Pz}})$, a faithful normal state on $\widetilde{Pz}$, with a fixed faithful normal tracial state $\tau := \chi_z \bar{\otimes}\tau_0$ on $\widetilde{Qz} = Qz\bar{\otimes}\lambda(\mathbb{R})''$. Note that $\lambda(t) u_z = \sigma_t^{\chi_z\circ(E_R\!\upharpoonright_{Pz})}(u_z)\lambda(t) = u_z\lambda(t)$ for all $t \in \mathbb{R}$. Thus, for any finite sum $x = \sum_k x_k \lambda(t_k) \in \widetilde{Pz}$ with $x_k \in Pz$ we have $\varphi_z(u_z x) = \sum_k \tau(\Psi(u_z E_{(Qz)'\cap Pz}(x_k))\lambda(t_k)) 
= \sum_k \chi_z(\Psi(u_z E_{(Qz)'\cap Pz}(x_k)))\tau_0(\lambda(t_k)) = \sum_k \chi_z(\Psi(E_{(Qz)'\cap Pz}(x_k)u_z))\tau_0(\lambda(t_k)) = \sum_k \varphi_z(x_k u_z \lambda(t_k)) = \varphi_z(x u_z)$. It follows that $u_z$ falls in $(\widetilde{Pz})_{\varphi_z}$. Clearly $\widetilde{E}_R(u_z^k) = (\widetilde{E_R\!\upharpoonright_{Pz}})(u_z^k) = E_R(u_z^k) = \Psi(u_z^k) = 0$ for all $k \neq 0$. Set $\varphi(x) := \frac{1}{2}(\varphi_z(xz) + \varphi_{z^\perp}(xz^\perp))$ for $x \in \widetilde{P}$, and then $\varphi$ becomes a faithful normal state on $\widetilde{P}$ and satisfies $\varphi\circ \widetilde{E}_R = \varphi$, implying the desired condition (a), since $\widetilde{R}$ is commutative. For each $n \in \mathbb{N}$ with $n \geq 2$ we define the unitary $u_n := u_z \oplus u_{z^\perp,n} \in \widetilde{Pz}\oplus\widetilde{Pz^\perp} = \widetilde{P}$, and thus $\widetilde{E}_R(u_n^k) = (\widetilde{E_R\!\upharpoonright_{Pz}})(u_z^k)\oplus(\widetilde{E_R\!\upharpoonright_{Pz^\perp}})(u_{z^\perp,n}^k) = 0$ for all $1 \leq k \leq n-1$. Hence the desired condition (b) is immediate as in (1) from the fact that $\widetilde{E}_Q = \widetilde{E}_Q\circ\widetilde{E}_R$ and $E^\varphi_{\widetilde{Q}} = E^\varphi_{\widetilde{Q}}\circ\widetilde{E}_R$ (the latter follows from $\varphi\circ \widetilde{E}_R = \varphi$). The final assertion is shown in the exactly same way (but easier) as above.       
\end{proof} 

We will give two applications of Proposition \ref{P-3.5}. The latter is a straightforward generalization of both \cite[Theorem 3.7]{Ueda:AdvMath11} and \cite[Proposition 3.1]{Ueda:MRL}. Remark that the former reproves the assertions (1), (4) in Corollary \ref{C-4.4} without any use of the technologies provided in \S\S3.1--3.2. 

\begin{theorem}\label{T-4.8} Assume that $M_1$ diffuse, $N$ of type I and $M_2 \supseteq N$ entirely non-trivial. Let $z \in \mathcal{Z}(N)$ be the unique projection such that $Nz$ is diffuse and $Nz^\perp$ atomic, and assume further that $(M_1)c_z^{M_1}$ has no type I direct summand when $z \neq 0$ {\rm(}i.e., this last assumption is fulfilled if $M_1$ has no type I direct summand{\rm)}.  Then $(\widetilde{M})_\omega = \big(\widetilde{M}\big)'\cap \big(\widetilde{M}\big)^\omega = \big(\widetilde{M}\big)'\cap\mathcal{Z}(\widetilde{N})^\omega$. In particular, $\widetilde{M}$ and hence $M$ itself are non-amenable. If $M$ is additionally assumed to be semifinite, then $M_\omega = M' \cap M^\omega = M' \cap \mathcal{Z}(N)^\omega$ also holds.  
\end{theorem} 

After the completion of the main part of the present work we learned that Houdayer and Vaes  have also independently been obtained a similar (but not same) result as above under different assumptions with different (and simpler) methods. See \cite[Theorem 5.8]{HoudayerVaes:Preprint12}. More on this will be discussed at the end of this section. 

\begin{proof} Note that $(\widetilde{N} \supseteq N) \cong (N\bar{\otimes}\lambda(\mathbb{R})'' \supseteq N\bar{\otimes}\mathbb{C}1)$. Since $N$ is of type I, one can choose an abelian $f \in N^p$ ($\subset \widetilde{N}^p$) with $c_f^{\widetilde{N}} = 1$. Let us first prove: 
\begin{equation}\label{Eq-4.1}
f\big(\widetilde{N}\big)^\omega f = \mathcal{Z}(\widetilde{N})^\omega f.  
\end{equation}
For each $x \in \widetilde{N}^\omega$ with representative $(x(m))_m$ one has $fxf = [(fx(m)f)_m]$, and for every $m$ there is a unique $z(m) \in \mathcal{Z}(\widetilde{N})$ with $fx(m)f=z(m)f$. By $c_f^{\widetilde{N}} = 1$ the mapping $x' \in \widetilde{N}' \mapsto x' f \in \widetilde{N}' f$ gives a bijective normal $*$-homomorphism (thus $\Vert-\Vert_\infty$-preserving), and hence $(z(m))_m$ defines $z \in \mathcal{Z}(\widetilde{N})^\omega$. Consequently we get $fxf = zf \in \mathcal{Z}(\widetilde{N})^\omega f$. 

By Lemma \ref{L-4.6} together with \eqref{Eq-2.5} we have the identification 
\begin{equation}\label{Eq-4.2} 
\big(\widetilde{fMf}, \widetilde{E\!\upharpoonright_{fMf}}\big) = \big(\widetilde{fM_1 f}, \widetilde{E_1\!\upharpoonright_{fM_1 f}}\big)\star_{\widetilde{fNf}} \big(\widetilde{fM_2 f}, \widetilde{E_2\!\upharpoonright_{fM_2 f}}\big). 
\end{equation}  
Let $c \in \mathcal{Z}(M_1)$ be the central support projection of the type I direct summand of $M_1$. Then $e = cf$ is that of $f M_1 f$ too, and $fNf e = \mathcal{Z}(N)fe$ must be atomic (or $0$ if $e = 0$) by the assumption here. In fact, if this was not the case, then $\mathcal{Z}(N) c_e^N \cong \mathcal{Z}(N)e = \mathcal{Z}(N)fe$ is not atomic, and hence $z c_e^N \neq 0$, i.e., $ze \neq 0$, implying $c_z^{M_1} c \geq z c \geq ze \neq 0$, a contradiction to that $M_1 c_z^{M_1}$ has no type I direct summand. Therefore, by Lemma \ref{L-4.7} (3) we can apply Proposition \ref{P-3.5} to \eqref{Eq-4.2} and thus any $x \in \big(\widetilde{fMf}\big)' \cap \big(\widetilde{fMf}\big)^\omega$ and any $y \in \big(\widetilde{fM_2 f}\big)^\circ$ must satisfy that
\begin{equation}\label{Eq-4.3}  
(\widetilde{E_2\!\upharpoonright_{fM_2 f}})(y^* y)(x-(\widetilde{E_{fM_1 f}})^\omega(x)) = 0,  
\end{equation} 
where $E_{fM_1 f}$ is the unique conditional expectation from $fMf$ onto $fM_1 f$ determined as \eqref{Eq-2.2}. Note that $\widetilde{fM_2 f} \supseteq \widetilde{fNf}$ with $\widetilde{E_2\!\upharpoonright_{fM_2 f}}$ contains $fM_2 f \supseteq fNf$ with $E_2\!\upharpoonright_{f M_2 f}$ canonically. Hence, by Lemma \ref{L-4.2} (2), (3) one can find a family $\{y_i\}_{i \in I}$ in $(f M_2 f)^\circ$ in such a way that $\sum_{i\in I} s(E_2(y_i^* y_i)) = f$ ($=1_{fNf}$). Therefore, it follows from \eqref{Eq-4.3} as in the proof of Theorem \ref{T-4.3} that $x = \big(\widetilde{E_{fM_1 f}}\big)^\omega(x) \in \big(\widetilde{fM_1 f}\big)^\omega$. Consequently $\big(\widetilde{fMf}\big)' \cap \big(\widetilde{fMf}\big)^\omega = \big(\widetilde{fMf}\big)' \cap \big(\widetilde{fM_1 f}\big)^\omega$. In the same way as in the proof of Theorem \ref{T-4.3}, we see, by using the above $\{y_i\}_{i\in I}$ again and the free independence between $\big(\widetilde{fM_1 f}\big)^\omega$ and $\big(\widetilde{fM_2 f}\big)^\omega$, that 
\begin{equation}\label{Eq-4.4} 
\big(\widetilde{fMf}\big)' \cap \big(\widetilde{fMf}\big)^\omega = \big(\widetilde{fMf}\big)' \cap \big(\widetilde{fNf}\big)^\omega.
\end{equation}   
Choose a faithful normal semifinite trace $\mathrm{Tr}_N$ on $N$, and $\widetilde{M} \supseteq \widetilde{M}_k\, (k=1,2)\, \supseteq \widetilde{N}$ are realized as 
$\widetilde{M} = M\rtimes_{\sigma^{\mathrm{Tr}_N\circ E}}\mathbb{R} \supseteq \widetilde{M}_k = M_k \rtimes_{\sigma^{\mathrm{Tr}_N\circ E_k}}\mathbb{R} \supseteq \widetilde{N} = N\rtimes_{\sigma^{\mathrm{Tr}_N}}\mathbb{R}$. Since $\sigma_t^{\mathrm{Tr}_N\circ E}(f) = f$ for all $t \in \mathbb{R}$, $f\widetilde{M}f \supseteq f\widetilde{M}_k f \supseteq f\widetilde{N}f$ are naturally identified with $\widetilde{fMf} \supseteq \widetilde{fM_k f} \supseteq \widetilde{fNf}$. Hence \eqref{Eq-4.4} and \eqref{Eq-4.1} imply that  
\begin{align}\label{Eq-4.5} 
\big(\widetilde{M}\big)'f \cap f\big(\widetilde{M}\big)^\omega f 
= \big(\widetilde{M}\big)'f \cap f\widetilde{N}^\omega f = \big(\widetilde{M}\big)'f \cap \mathcal{Z}(\widetilde{N})^\omega f.
\end{align} 

Let $\pi_f$ be the normal surjective $*$-homomorphism $x \in \big(\widetilde{M}\big)'\cap\big(\widetilde{M}\big)^\omega \mapsto xf \in \big(\widetilde{M}'\cap\big(\widetilde{M}\big)^\omega\big)f = \big(\widetilde{M}\big)'f \cap (f\big(\widetilde{M}\big)^\omega f)$ (c.f.~\cite[Lemma 4.1 (i)]{Voeden:PLMS73}), which is also injective due to $c_f^{\widetilde{N}} = 1$ (and hence $c_f^{\widetilde{M}} = 1$ too). By \eqref{Eq-4.5} we have $\big(\widetilde{M}\big)'\cap\big(\widetilde{M}\big)^\omega = \pi_f^{-1}\big(\big(\widetilde{M}\big)'f \cap \mathcal{Z}(\widetilde{N})^\omega f\big)$. As in the proof of Lemma \ref{L-4.2} one can choose partial isometries $\{v_i\}_{i\in I}$ in $\widetilde{N}$ so that $\sum_{i\in I}v_i^* v_i = c_f^{\widetilde{N}} = 1$ and $v_i v_i^* \leq f$ for all $i \in I$. Then, if $x = zf \in \big(\widetilde{M}\big)'f\cap\mathcal{Z}(\widetilde{N})^\omega f$ with $z \in \mathcal{Z}(\widetilde{N})^\omega$, then we have 
$yz = \sum_{i_1,i_2\in I} v_{i_1}^* v_{i_1} y z v_{i_2}^* v_{i_2} = \sum_{i_1,i_2\in I} v_{i_1}^* (v_{i_1} y v_{i_2}^*) x v_{i_2} = \sum_{i_1,i_2\in I} v_{i_1}^*  x (v_{i_1} y v_{i_2}^*) v_{i_2} = \sum_{i_1,i_2\in I} v_{i_1}^* v_{i_1} zy v_{i_2}^* v_{i_2} = zy$ for $y \in \widetilde{M}$, implying $z \in \big(\widetilde{M}\big)'\cap\mathcal{Z}(\widetilde{N})^\omega$. Hence $\big(\widetilde{M}\big)'f \cap \mathcal{Z}(\widetilde{N})^\omega f = \big(\big(\widetilde{M}\big)'\cap\mathcal{Z}(\widetilde{N})^\omega\big)f$. Consequently $\big(\widetilde{M}\big)'\cap\big(\widetilde{M}\big)^\omega = \pi_f^{-1}\big(\big(\widetilde{M}\big)'f \cap \mathcal{Z}(\widetilde{N})^\omega f\big) = \pi_f^{-1}\big(\big(\big(\widetilde{M}\big)'\cap\mathcal{Z}(\widetilde{N})^\omega\big)f\big) = \big(\widetilde{M}\big)'\cap\mathcal{Z}(\widetilde{N})^\omega$. Since $\big(\widetilde{M}\big)'\cap\big(\widetilde{M}\big)^\omega = \big(\widetilde{M}\big)'\cap\mathcal{Z}(\widetilde{N})^\omega$ is commutative, it must equal $(\widetilde{M})_\omega$ as observed in \cite[(8) in page 360]{Ueda:TAMS03}. 

The final assertion is also shown in the exactly same way as above by using the final assertion in Lemma \ref{L-4.7} (3), since there is a faithful normal semifinite trace $\mathrm{Tr}_N$ on $N$ so that $\mathrm{Tr}_N\circ E_k$ ($k=1,2$) are traces again thanks to Corollary \ref{C-4.4} (3). 
\end{proof}

\begin{remark}\label{R-4.9}{\rm 
The same type argument as in Theorem \ref{T-4.3} (3) works for constructing a faithful normal state $\chi$ on $N$ with $\sigma_T^{\chi\circ E} = \mathrm{Id}$ with $T = -2\pi/\log\lambda$, $0 < \lambda < 1$, when $M$ is known to be a factor of type III$_\lambda$ under the same set of assumptions as in Theorem \ref{T-4.8}. Hence the discrete core of such $M$ can also be written as an amalgamated free product von Neumann algebra of the same form as the continuous core, and an analogous formula for its asymptotic centralizer holds. In particular, the discrete core of such a factor of type III$_\lambda$ is an $\infty$-amplification of a non-strongly stable type II$_1$ factor. Further and more detailed discussions related to this aspect will be given elsewhere. 
}   
\end{remark}
 
\begin{theorem}\label{T-4.10} If $M_1$ is diffuse, $N$ of atomic type I and $M_2 \supseteq N$ entirely non-trivial, then the following hold true{\rm:}
\begin{itemize} 
\item[(1)] $M_\omega = M'\cap M^\omega = M' \cap \mathcal{Z}(N)$ {\rm(}$= \mathcal{Z}(M)${\rm)}. Hence $M$ does never have no type III$_0$ direct summand  {\rm(}see \cite[Theorem 2.12]{Connes:JFA74}{\rm)}, and becomes full in the sense of Connes {\rm \cite{Connes:JFA74}} under the separability of preduals.  
\item[(2)] The Connes $\tau$-invariant $\tau(M)$ {\rm(}see {\rm \cite{Connes:JFA74}}{\rm)} is determined under the separability of preduals as follows. Let $\chi$ be a faithful normal state on $N$. Then $t_m \longrightarrow 0$ in $\tau(M)$ as $m\rightarrow \infty$ if and only if there is a unitary $w \in N$ so that $\sigma_{t_m}^{\chi\circ E} \longrightarrow \mathrm{Ad}w$ in $\mathrm{Aut}(M)$ as $m \rightarrow \infty$. 
\end{itemize} 
\end{theorem} 
\begin{proof} (1) This is proved along the same line as in the proof of Theorem \ref{T-4.8} by using only Lemma \ref{L-4.7} (2) instead together with a well-known fact $\mathcal{Z}(N) = \mathcal{Z}(N)^\omega$ due to the assumption that it is atomic. 

(2) We can write $N = \sum_{i \in I}^\oplus B(\mathcal{H}_i)$. Looking at this structure with the given $\chi$ we can choose a collection $\{e_i\}_{i\in I}$ of abelian projections in $N$ with $\sum_{i\in I} e_i = 1$ such that for each $i \in I$ there is a larger abelian $f_i \in N^p$ so that $e_i \leq f_i$, $c_{f_i}^N = 1$ and $\sigma_t^\chi(f_i) = f_i$ ($t \in \mathbb{R}$). Assume that $t_m \longrightarrow 0$ in $\tau(M)$ as $m \rightarrow \infty$. Then there is a sequence $(u_m)_m$ of unitaries in $M$ such that $\mathrm{Ad}u_m\circ\sigma_{t_m}^{\chi\circ E} \longrightarrow \mathrm{id}$ in $\mathrm{Aut}(M)$ as $m\rightarrow \infty$. As observed in the proof \cite[Proposition 3.1]{Ueda:MRL} the $(u_m)_m$ defines a unitary $u \in M^\omega$, and clearly $u f_i = f_i u$ for all $i \in I$. Hence $f_i u$ defines a unitary in $f_i M^\omega f_i = (f_i M f_i)^\omega$, and we denote it by $u_i$ for simplicity. Since $f_i M_1 f_i$ is still diffuse, looking at $f_i M_1 f_i \supseteq f_i N f_i = \mathcal{Z}(N)f_i$ one can choose a faithful normal state $\varphi$ on $f_i M_1 f_i$ as in Lemma \ref{L-4.7} (2). Set $\hat{\varphi}(x) := \varphi(f_i x f_i) + \chi\circ E_1(f_i^\perp x f_i^\perp)$, $x \in M_1$, which becomes a faithful normal positive linear functional on $M_1$. Clearly $f_i \in (M_1)_{\hat{\varphi}}$ and thus $f_i [D\chi\circ E_1:D\hat{\varphi}]_t = [D\chi\circ {E_1}\!\upharpoonright_{f_i M_1 f_i} : D\varphi]_t$ for all $t \in \mathbb{R}$ by the uniqueness part of Connes's Radon-Nikodym cocycle theorem. As observed in the proof of \cite[Proposition 3.1]{Ueda:MRL} again the sequence $v_m := [D\chi\circ E_1 : D\hat{\varphi}]_{t_m}$ defines a unitary $v \in M_1^\omega$ and also the sequence $f_i v_m = v_m f_i$ does a unitary $v_i \in f_i M_1^\omega f_i = (f_i M_1 f_i)^\omega$. Since $\hat{\varphi}\circ {E_{M_1}}\!\upharpoonright_{f_i M f_i} = \varphi\circ({E_{M_1}}\!\upharpoonright_{f_i M f_i})$, we have $yu_i v_i = yuv = \big[(yu_m v_m)_m\big] = \big[ (u_m v_m \sigma_{t_m}^{\varphi\circ({E_{M_1}}\!\upharpoonright_{f_i M f_i})}(y))_m\big] = uvz = u_i v_i z$ for $y \in (f_i M_2 f_i)^\circ$ with $z=\big[(\sigma_{t_m}^{\varphi\circ({E_{M_1}}\!\upharpoonright_{f_i M f_i})}(y))_m\big] \in (f_i M f_i)^\omega = f_i M^\omega f_i$ in the identification $(f_i M f_i, E\!\upharpoonright_{f_i M f_i}) = (f_i M_1 f_i, {E_1}\!\upharpoonright_{f_i M_1 f_i}) \star_{f_i N f_i} (f_i M_2 f_i, {E_2}\!\upharpoonright_{f_i M_2 f_i})$ provided by Lemma \ref{L-4.6}. By Proposition \ref{P-3.5} we get $({E_2}\!\upharpoonright_{f_i M_2 f_i})(y^* y)(u_i v_i - ({E_{M_1}}\!\upharpoonright_{f_i M f_i})^\omega(u_i v_i)) = 0$ for $y \in (f_i M_2 f_i)^\circ$. By using Lemma \ref{L-4.2} (2), (3) twice as in the proof of Theorem \ref{T-4.8} we can prove firstly that $u_i v_i \in (f_i M_1 f_i)^\omega = f_i M_1^\omega f_i$, secondly that $u_i \in f_i M_1^\omega f_i$ (since $v_i \in f_i M_1^\omega f_i$), and finally that $u_i \in f_i N^\omega f_i = \mathcal{Z}(N)^\omega f_i = \mathcal{Z}(N)f_i$. Therefore, $u = \sum_{i \in I} e_i u = \sum_{i \in I} e_i f_i u = \sum_{i\in I} e_i u_i \in N$. Letting $w := u^* \in N^u$ we have $\mathrm{Ad}w^*\circ\sigma_{t_m}^{\chi\circ E} \longrightarrow \mathrm{id}$ in $\mathrm{Aut}(M)$ as $m\rightarrow\infty$.  
\end{proof}

The next proposition shows that Proposition \ref{P-3.5} is still useful beyond the case where $N$ is of type I or even semifinite. The proof goes along the same line as that of Theorem \ref{T-4.8} but is easier than it. Hence the proof is left to the reader. 

\begin{proposition}\label{P-4.11} Assume that there is a faithful normal state $\varphi$ on $M_1$ satisfying the following conditions{\rm:} 
\begin{itemize} 
\item[(a)] $\sigma_t^\varphi(N) = N$ for all $t \in \mathbb{R}$.  
\item[(b)] For every $n \in \mathbb{N}$ with $n \geq 2$ there are unitaries $u_k = u_k^{(n)}
,v_k = v_k^{(n)} \in (M_1)_\varphi$, $0 \leq k \leq n-1$, such that $E_1(u_{k_1}^* u_{k_2}) = E_N^\varphi(v_{k_1}^* v_{k_2}) = 0$ for all $0 \leq k_1 \neq k_2 \leq n-1$, where $E_N^\varphi$ denotes the unique $\varphi$-preserving conditional expectation from $M_1$ onto $N$, whose existence follows from {\rm(a)} and Takesaki's criterion.  
\end{itemize} 
Assume also that $M_2 \supseteq N$ is entirely non-trivial. Then $M' \cap M^\omega = M' \cap N^\omega$ holds. Moreover, if it is further assumed that $N$ is finite, then $M_\omega = M'\cap M^\omega = M' \cap N_\omega$.  
\end{proposition} 

It is easy to confirm that the $(M_1,E_1)$ in Corollary \ref{C-4.5} satisfies the  assumption of Proposition \ref{P-4.11}. Thus $M'\cap M^\omega = M'\cap N^\omega$ holds under the set of assumptions in Corollary \ref{C-4.5}. 
%

Assume that $M_1$ is a von Neumann algebra with separable predual and that $N$ is a Cartan subalgebra in $M_1$. It was proved in \cite[Lemma 4.2]{Ueda:PacificJMath99} that if $M_1$ is further assumed to be a non-type I factor, then there are a faithful normal state $\varphi$ on $M_1$ with $\varphi\circ E_1 = \varphi$ and a unitary $u \in (M_1)_\varphi$ such that $E_1(u^k) = 0$ as long as $k \neq 0$. The same assertion can indeed be proved even when $M_1$ is further assumed only to have no type I direct summand (i.e., without being a factor). The proof is similar to \cite[Lemma 4.2]{Ueda:PacificJMath99} but tedious based on disintegration. Hence such $(M_1,E_1)$ satisfies the assumption of Proposition \ref{P-4.11}.

\begin{remark}\label{R-4.12} {\rm 
Almost all the results obtained above have appropriate `HNN variants' thanks to tricks given in \cite{Ueda:IllinoisJMath08}. Here it should be emphasized that our results so far essentially need assumptions for only one free component. The notion of HNN extensions of von Neumann algebras as well as their basic properties including their modular theoretic aspects were established in \cite{Ueda:JFA05}.}
\end{remark} 

In closing of this section we discuss one of  Houdayer and Vaes's results \cite[Theorem 5.8]{HoudayerVaes:Preprint12}. This part of the present paper is added after receiving a draft of \cite{HoudayerVaes:Preprint12} in order to point out only one consequence obtained from this and that papers without any new idea. Therefore, some facts provided in \cite{HoudayerVaes:Preprint12} are necessary below. The original aim of the present work is to provide amalgamated free product counterparts of the results in \cite[\S3]{Ueda:AdvMath11}. One issue to do so is how to formulate a suitable assumption saying that $M_1$ is `diffuse relative to $N$' which corresponds to that $M_1$ is diffuse when $N=\mathbb{C}1$. The requirement for $M_1 \supseteq N$ in Theorem \ref{T-4.3} seems to be one strong form of them without any restriction on $N$, but it seems not so easy to check it in general. Thus we propose the requirement for $M_1 \supseteq N$ in Corollary 4.4 and Theorem 4.8 as such a candidate in the special case when $N$ is of type I. However a more sophisticated one in the special case seems to be that $M_1 \supseteq N$ has no trivial corner, which is proposed in \cite[\S5]{HoudayerVaes:Preprint12} by a different motivation. In fact, Houdayer and Vaes \cite[Theorem 5.8]{HoudayerVaes:Preprint12} give a factoriality and non-amenability result under the set of assumptions that both $M_k \supseteq N$, $k=1,2$, have no trivial corner and that $N$ is of type I, and establish their primeness result under the same set of assumptions. Here an inclusion $P \supseteq Q$ of von Neumann algebras is said to have no trivial corner if $pPp \neq Qp$ for any non-zero projection $p \in Q'\cap P$. Any exact general relationship between theirs and ours is not immediately clear. However the proof of Theorem \ref{T-4.8} and general properties on inclusions without trivial corner provided in \cite[\S\S5.1]{HoudayerVaes:Preprint12} altogether immediately give an improvement of \cite[Theorem 5.8]{HoudayerVaes:Preprint12}, though it is not immediately clear whether the primeness result in \cite[Theorem E]{HoudayerVaes:Preprint12} holds or not under the new set of assumptions. 

\begin{theorem}\label{T-4.13} If $M_1 \supseteq N$ has no trivial corner, $N$ is of type I and $M_2 \supseteq N$ entirely non-trivial, then the following hold true{\rm:} 
\begin{itemize} 
\item[(1)] $\mathcal{Z}(M) = \mathcal{Z}(M_1)\cap\mathcal{Z}(M_2)\cap\mathcal{Z}(N)$. 
\item[(2)] $\mathcal{Z}(\widetilde{M}) = \mathcal{Z}(\widetilde{M}_1)\cap\mathcal{Z}(\widetilde{M}_2)\cap\mathcal{Z}(\widetilde{N})$. 
\item[(3)] $(\widetilde{M})_\omega = \big(\widetilde{M}\big)'\cap\big(\widetilde{M}\big)^\omega = \big(\widetilde{M}\big)'\cap\mathcal{Z}(\widetilde{N})^\omega$. 
\end{itemize} 
In particular, {\rm(3)} explains that $M$ does never become amenable.  
\end{theorem}  
\begin{proof} It is trivial that (3) $\Rightarrow$ (2) $\Rightarrow$ (1), see e.g.~the proof of \cite[Theorem 5.2]{Ueda:JFA05} for (3) $\Rightarrow$ (2) and \cite[Theorem X.II.1.1]{Takesaki:Book2} for (2) $\Rightarrow$ (1). Thus it suffices to prove only (3). The line of the proof below is exactly identical to that of Theorem \ref{T-4.8}, and thus we keep the notations there. In fact, only one modification is sufficient. By \cite[Lemma 5.2, Proposition 5.5]{HoudayerVaes:Preprint12} the inclusion $\widetilde{fM_1 f} \supseteq \widetilde{fNf}$ also has no trivial corner. Then it suffices to prove the exactly same assertion as in Lemma \ref{L-4.7} (1) with replacing the assumption that $P$ has no type I direct summand by that $P \supseteq Q$ has no trivial corner. In fact, by using this new assertion instead of Lemma \ref{L-4.7} (3) one gets the same equation \eqref{Eq-4.3} and the rest of the proof there works well.   

Let $P \supseteq Q$ be an inclusion of von Neumann algebras without trivial corner. Assume that $Q$ is commutative, $P$ has a faithful normal semifinite trace $\mathrm{Tr}_P$ and there is a faithful normal conditional expectation $E_Q : P \rightarrow Q$ satisfying $\mathrm{Tr}_P \circ E_Q = \mathrm{Tr}_P$. As in the proof of Lemma \ref{L-4.7} (1) we choose the $q_m$'s and $\chi$. Then we apply \cite[Lemma 5.4 (3)]{HoudayerVaes:Preprint12} (note that it holds without assuming the separability of preduals, see Lemma \ref{L-4.14} below) with $q = p := q_m$ and get a unitary $u_m \in q_m P q_m$ satisfying that $E_Q(u_m^k) = 0$ as long as $k \neq 0$. Letting $u := \sum_{m \in \mathbb{N}} u_m$ we have $u \in P_{\chi\circ E_Q}$ and $E_Q(u^k) = 0$ as long as $k \neq 0$. Hence we are done. 
\end{proof}   

As remarked in \cite[Lemma 5.3]{HoudayerVaes:Preprint12} the next lemma immediately follows from Rohlin's general theorem on Lebesgue spaces under the separability of preduals. Thus only the advantage of the proof below is no use of disintegration; hence the separability of preduals is not necessary in  \cite[Lemma 5.4]{HoudayerVaes:Preprint12}. Although it is a rather minor point, we do give it for the sake of completeness. 

\begin{lemma}\label{L-4.14} Let $B \supseteq A$ be {\rm(}unital{\rm)} inclusion of commutative $\sigma$-finite von Neumann algebras with a faithful normal conditional expectation $E_A : B \rightarrow A$. If $Bf \neq Af$ for any nonzero projection $f \in B$, then there is a unitary $u \in B$ such that $E_A(u^k) = 0$ as long as $k \neq 0$. 
\end{lemma} 
\begin{proof} Choose non-zero $f \in B^p$. Since $Bf \neq Af$, there is $x \in B$ such that $x \not\in Af$ and $0 \leq x \leq f$. Since $E_A(x) \leq E_A(f)$, one can choose a positive contraction $c \in A$ so that $c E_A(f) = E_A(x)$ (since $A$ is commutative). Letting $y := x - cf \in Bf$ we have $y = y^* \neq 0$ (due to $x \not\in Af$) and $E_A(y) = 0$. Therefore, an idea given in the proof of \cite[Lemma 2.1]{CameronFangMukherjee:Preprint11} enables us to construct projections $e_{(\varepsilon_1,\dots,\varepsilon_n)} \in B$, $n \in \mathbb{N}$, $\varepsilon_k \in \{1,2\}$, in such a way that $e_{(\varepsilon_1,\dots,\varepsilon_n)} = e_{(\varepsilon_1,\dots,\varepsilon_n,1)}+e_{(\varepsilon_1,\dots,\varepsilon_n,2)}$ and $E_A(e_{(\varepsilon_1,\dots,\varepsilon_n)}) = \frac{1}{2^n}1$. The proof is done by induction. Assume that we have chosen up to $n$-th stage. Set $\Lambda_e := \{ x = x^* \in Be\,|\,\Vert x \Vert_\infty \leq 1, E_A(x) = 0\}$ with $e := e_{(\varepsilon_1,\dots,\varepsilon_n)}$. It is a $\sigma$-weakly compact convex subset, and thus has sufficiently many extremal points due to the Krein--Milman theorem. Let $a \in \Lambda_e$ be an extremal point. Then it suffices to prove $a = 2e_0 - e$ for some $e_0 \in B^p$ with $e_0 \leq e$, since it clearly implies that $E_A(e_0) = \frac{1}{2}E_A(e)$. On contrary, suppose that it is not the case. By the spectral decomposition of $a$ one can find $\delta>0$ and non-zero $f \in B^p$ in such a way that $f \leq e$ and $-(1-\delta)f \leq af \leq (1-\delta)f$. By what we have shown above, there is a non-zero $y = y^* \in Bf$ such that $-\delta f \leq y \leq \delta f$ and $E_A(y) = 0$, and hence $a+y, a-y \in \Lambda_e$ and $a = \frac{1}{2}(a+y)+\frac{1}{2}(a-y)$, a contradiction. Thus $e_{(\varepsilon_1,\dots,\varepsilon_n,1)} := e_0$ and $e_{(\varepsilon_1,\dots,\varepsilon_n,2)} := e-e_0$ become desired ones in $(n+1)$-th stage. Hence we have proved the claim. Let $(C,\omega)$ be the von Neumann algebraic infinite tensor product of $\mathbb{C}\oplus\mathbb{C}$ with equal weights $\{1/2,1/2\}$. Once passing GNS representations one can construct an injective normal $*$-homomorphism from $C\bar{\otimes}A$ into $B$ which intertwines $\omega\bar{\otimes}\mathrm{id}_A$ and $E_A$. Hence the desired assertion follows, since $(C,\omega) \cong (L(\mathbb{Z}),\tau_{\mathbb{Z}})$ thanks to \cite[Theorem III.1.22]{Takesaki:Book1}.    
\end{proof} 

The entire non-triviality of an inclusion $P \supseteq Q$ of von Neumann algebras is nothing but just the non-triviality of $P$ when $Q = \mathbb{C}1$, and hence Theorem \ref{T-4.13} is no longer true under assuming only that $M_1 \supseteq N$ is entirely non-trivial instead. In fact, the plain free product of two 2-dimensional algebras with suitable states provides a counter example, see \cite{Ueda:AdvMath11} for suitable references therein. Finally we conjecture that Corollary 4.4, especially a strong kind of irreducibility $((M_1)_\varphi)' \cap M \subseteq M_1$ for some faithful normal state $\varphi$, should also hold under the same set of assumptions of Theorem \ref{T-4.13}. This is rather technical, but such a property may have some potential in further analysis. We will consider it in future work beyond the case where $\mathcal{Z}(M) = \mathcal{Z}(M_1)\cap\mathcal{Z}(M_2)\cap\mathcal{Z}(N)$ need not hold.       

\section*{Acknowledgment} I thank Professors Cyril Houdayer and Stefaan Vaes for several fruitful conversations in Dec.~2011 and in Jan.~2012 and also for sending us a draft of \cite{HoudayerVaes:Preprint12} prior to putting it on the ArXiv.  

}

\end{document}